\documentclass[a4paper,draft]{amsproc}
\usepackage{amssymb}
\usepackage[hyphens]{url} 
\usepackage{bussproofs}

\theoremstyle{plain}
 \newtheorem{thm}{Theorem}[section]
 \newtheorem{prop}{Proposition}[section]
 \newtheorem{lem}{Lemma}[section]
 \newtheorem{cor}{Corollary}[section]
\theoremstyle{definition}
 \newtheorem{exm}{Example}[section]
 \newtheorem{dfn}{Definition}[section]
 \newtheorem{con}{Convention}[section]
\theoremstyle{remark}
 \newtheorem{rem}{Remark}[section]

 \numberwithin{equation}{section}

\renewcommand{\le}{\leqslant}\renewcommand{\ge}{\geqslant}

\renewcommand{\setminus}{\smallsetminus}
\newcommand{\Ll}{\mathcal{L}_1}
\newcommand{\Var}{\mathit{Var}_1}
\newcommand{\Cost}{\mathit{Con}}
\newcommand{\Pre}{\mathit{Pred}}
\newcommand{\Fun}{\mathit{Fun}}
\newcommand{\Ter}{\mathit{Term}}
\newcommand{\Frml}{\mathit{Form}_1}
\newcommand{\Mod}{\mathbf{M}_1}
\newcommand{\Sat}{\mathit{Sat}}
\newcommand{\Leg}{\mathrm{Leg}}
\newcommand{\adown}{a_{\downarrow}}
\newcommand{\aup}{a_{\uparrow}}
\newcommand{\segue}{\Rightarrow}
\newcommand{\Sub}{\mathrm{Sub}}
\newcommand{\Subg}{\mathrm{Sub}^{\ast}}
\newcommand{\dg}[1]{\#(#1)}
\newcommand{\wt}[1]{\mathrm{w}(#1)}
\newcommand{\LD}{\Ll(M)}
\newcommand{\bd}[1]{\vdash_{#1}}

\title[Sequent-style tableaux for first-order logic]{Sequent-style tableaux for first-order logic: structural analysis, cut admissibility, and the correspondence with $\mathrm{LK}$}

\subjclass[2020]{Primary 03F03; Secondary 03B10, 03F05}

\keywords{Sequent-style tableaux, block calculus, cut admissibility, subformula property, Hintikka sets, proof theory}

\author[Cuconato]{\bfseries Simone Cuconato}

\address{
Department of Physics \\
University of Calabria \\
Arcavacata di Rende (CS) \\
Italy}
\email{simone.cuconato@unical.it}

\begin{document}

{\begin{flushleft}\baselineskip9pt\scriptsize
MANUSCRIPT
\end{flushleft}}
\vspace{18mm} \setcounter{page}{1} \thispagestyle{empty}

\begin{abstract}
We give a self-contained development of the first-order block calculus in unsigned sequent-style notation: each node of the refutation tree carries a finite block $\Pi = \Gamma \cup \lnot[\Delta]$, negation is governed by explicit rules, and a branch closes on a complementary pair of literals. The calculus is Smullyan's, and so in substance are the theorems; what is offered here is a different arrangement of them. The structural properties are established in the order of dependence familiar from G3-style sequent calculi: closure on arbitrary formulae is admissible, weakening and the substitution of parameters are admissible with preservation of the height, every rule is height-preserving invertible, and cut is admissible, the last being derived from the first three rather than conversely. Soundness, completeness under a fair strategy, countable compactness and the countable model property follow, together with a syntactic criterion under which every fair construction terminates. The correspondence is then proved, in both directions and with cut included, with Gentzen's LK in its usual presentation with explicit weakening, which requires lemmas on parameters that set-based Gentzen systems do not need.
\end{abstract}

\maketitle

\section{Introduction}\label{sec:intro}

Anyone who approaches for the first time Gentzen's sequent calculus \cite{Gentzen1935,mancosu2021,troelstra2000} and the analytic tableaux of Beth and Smullyan \cite{Smullyan1968} is struck by a singular fact. The two systems set out from different proof-theoretic intuitions -- the one builds a derivation from the bottom upwards by composing logical rules, the other searches for a model by refutation, decomposing a formula into a tree-shaped structure -- and yet they share the same combinatorial skeleton. The left introduction rules of $\mathrm{LK}$, read backwards, are the expansion rules of a tableau; a tableau branch that closes is the axiom $\alpha \segue \alpha$ of a sequent. This kinship is old -- it is systematized in Smullyan \cite{Smullyan1968}, in Fitting \cite{Fitting1996}, and in Negri and von Plato \cite{NegriVonPlato2001} -- and the calculus we study here is its notational realization: in place of labelling each node of the tree with a single well-formed formula, one labels each node with a finite set of formulae, which we call a \emph{block} and denote by $\Pi$. A block gathers in one list, devoid of internal structure, both the premisses and the negation of the conclusion one intends to refute: the sign $\segue$, rather than separating two ranks of formulae, is reabsorbed into the translation that turns a sequent $\Gamma \segue \Delta$ into the set $\Pi = \Gamma \cup \lnot[\Delta]$.

What we are after is not the notation but its proof-theoretic behaviour. A calculus with no cut rule is not, for that reason alone, a calculus in which cut is superfluous; the content of Gentzen's \emph{Hauptsatz} is the second statement, not the first, and it has to be proved. The spine of the paper is accordingly the structural analysis of Section~\ref{sec:strutturale}: closure on arbitrary complementary pairs is derivable from closure on literals (Lemma~\ref{lem:chiusura-gen}), weakening and the substitution of one parameter for another are admissible without increase of height (Lemmas~\ref{lem:sost} and \ref{lem:wk}), every rule is height-preserving invertible (Lemma~\ref{lem:inv}), and on this basis the rule of cut
\[
\frac{\Pi, \theta \qquad \Pi, \lnot\theta}{\Pi}
\]
is admissible (Theorem~\ref{thm:taglio}), by the double induction on the degree of the cut formula and on the sum of the heights that is standard for $\mathrm{G3}$-style calculi. Once cut is available on the side of blocks, the translation from $\mathrm{LK}$ of Section~\ref{sec:LK} no longer needs to assume its premiss cut-free, and the Hauptsatz for $\mathrm{LK}$ follows by transfer (Corollary~\ref{cor:hauptsatz}). None of these statements is new -- Section~\ref{sec:precedenti} says whose they are and where they may be found -- and the interest of the arrangement, if it has any, lies in the direction in which the implications run.

It is in the logic of predicates that a refutation calculus is tested, since there it must confront the potential infinity of the domain of quantification. Two new families of rules appear: the $\gamma$ rules, universal, which may be applied repeatedly along one and the same branch by instantiating on a constant $\adown$ already present, and the $\delta$ rules, existential, which require a fresh individual constant $\aup$, according to the mechanism by which one chooses a witness in the informal proof of an existential statement. Both families interact with the structural analysis: the persistence of the principal formula in the $\gamma$ rules is what plays the role of left contraction, and the freshness condition on $\aup$ is what makes the substitution lemma the right tool in the quantifier cases of cut admissibility.

Every derivation displayed below is written out in full, in the format fixed by Convention~\ref{conv:display}: the root block at the top, the rule and the parameter recorded beside each line, the entire current block at every node, no formula elided. The reader can check each step mechanically against the rules of Section~\ref{sec:regole}.

\subsection{Relation to earlier work}\label{sec:precedenti}

The calculus studied here is not new, and it is best to say at the outset whose it is.

Tableaux whose points are finite \emph{sets} of formulae rather than single formulae were introduced by Smullyan under the name \emph{block tableaux}: they are developed for propositional logic in \cite[Ch.~XI, \S1]{Smullyan1968} and for quantification theory in the section of the same chapter entitled ``Block tableaux and Gentzen systems for first-order logic'' \cite[Ch.~XI, \S2]{Smullyan1968}. His quantifier rules are ours. The rule $C$ adjoins to a point $S, \gamma$ the instance $\gamma(a)$ while retaining $\gamma$; the rule $D$ replaces $S, \delta$ by $S, \delta(a)$ for a parameter $a$ new to the branch. The translation between a closed block tableau and a derivation in a Gentzen system is his as well, and in the same form: one replaces each point $S$ by the corresponding sequent and reads the resulting tree in the opposite orientation \cite[pp.~107--109]{Smullyan1968}. So is the observation that closure may be restricted to atomic pairs without loss of completeness \cite[p.~110]{Smullyan1968}.

The structural results are likewise his. Smullyan calls a formula $X$ \emph{eliminable} when the existence of closed tableaux for $S \cup \{X\}$ and for $S \cup \{\bar X\}$ implies the existence of one for $S$, which is the admissibility of cut in the formulation of our Theorem~\ref{thm:taglio}; he observes that eliminability follows at once, and non-constructively, from the completeness theorem, and then proves it constructively in an abstract form that covers tableaux and Gentzen systems together \cite[Ch.~XII]{Smullyan1968}. His proof is a double induction on the degree of $X$ and on a weight assigned to closed tableaux, and its four hypotheses $P_1$--$P_4$ are, in order, the closure conditions of an analytic consistency property, invertibility with strictly decreasing weight, weakening with preservation of the weight, and the substitution of one parameter for another in a $\delta$ instance with preservation of the weight. Our Lemmas~\ref{lem:chiusura-gen}, \ref{lem:sost}, \ref{lem:wk} and \ref{lem:inv} are those four conditions, verified directly for the block calculus in unsigned notation. That the Hauptsatz for tableaux yields the Hauptsatz for Gentzen systems by constructive translation, which is our Corollary~\ref{cor:hauptsatz}, is stated on the same page \cite[p.~111]{Smullyan1968}. Compactness and the Skolem--L\"owenheim theorem for first-order tableaux are proved in \cite[Ch.~V, \S4]{Smullyan1968}, Craig's interpolation lemma and Beth's definability theorem in \cite[Ch.~XV]{Smullyan1968}. Presentations of tableaux on sets have circulated widely since; the development of tableau refutation procedures out of sequent calculi is traced by Fitting \cite{Fitting1999}. A block, read as the one-sided sequent $\Pi \segue$, is also closely related to the one-sided calculi of Sch\"utte and Tait \cite{Schuette1977, Tait1968}, from which it differs in that no negation normal form is presupposed and negation is governed by explicit rules, each of them analytic.

What the present paper offers is therefore not a new theorem but a different route to known ones, and a presentation in which three things are done differently.

The first is the order of dependence. Smullyan obtains the invertibility of the rules as a corollary of the Hauptsatz \cite[Ch.~XII, Cor.~3--5]{Smullyan1968}. Here invertibility is proved first and directly, by induction on the height of a closed tableau (Lemma~\ref{lem:inv}), and cut is proved from it together with the height-preserving admissibility of weakening and of substitution. This is the pattern by which the structural properties of $\mathrm{G3}$-style sequent calculi are established \cite{NegriVonPlato2001}, and Section~\ref{sec:G3} argues that it is the pattern that suits a calculus whose structural rules are absorbed exactly as theirs are.

The second is the sequent calculus on the other side of the correspondence. Smullyan's Gentzen systems take sequents to be pairs of finite \emph{sets} and their axioms to carry an arbitrary context, so that weakening and contraction never appear as rules. The correspondence with Gentzen's $\mathrm{LK}$ in its usual presentation, with the axiom $\alpha \segue \alpha$ and explicit rules of weakening and cut, requires the three lemmas on parameters and on the inversion of $\lnot$s proved in Section~\ref{sec:LK}, which have no counterpart on his side.

The third is the notation. The calculus is presented without signs, on blocks $\Pi = \Gamma \cup \lnot[\Delta]$, with explicit rules for each negated compound; this is the notation of \cite{cuconato2025metodi} and of the companion paper \cite{cuconato-proofteoria}, in which the propositional case is related to the sequent calculi $\mathrm{G0T}$ and $\mathrm{G3T}$ and to a natural deduction system with general elimination rules. The termination criterion of Theorem~\ref{thm:terminazione} is stated in those terms; it is a block-calculus form of the classical decidability of the Bernays--Sch\"onfinkel prefix class.

The remainder of the paper is organized as follows. Section~\ref{sec:preliminari} fixes syntax, semantics and the notion of block. Section~\ref{sec:regole} presents the calculus, the closure criterion and the fair strategy. Section~\ref{sec:strutturale} is the structural analysis, and culminates in cut admissibility. Section~\ref{sec:metateoria} proves soundness, completeness, compactness and the countable model property. Section~\ref{sec:analiticita} is devoted to analyticity and to termination. Section~\ref{sec:LK} establishes the correspondence with $\mathrm{LK}$, cut included, and draws the Hauptsatz. Section~\ref{sec:esempi} collects worked derivations, and some concluding remarks close the paper.

\section{Preliminaries}\label{sec:preliminari}

We fix here the language, the semantics and the translation of sequents into blocks. Nothing in this section is new, and the reader familiar with the standard apparatus may consult it only for the notation; two choices are, however, not the default ones, and they are made for reasons that will become visible later: the semantics evaluates sentences in a language expanded by names for the elements of the domain, and one individual constant is set aside once and for all.

\subsection{Syntax}

The first-order language $\Ll$ has a countable set $\Var = \{x_1, x_2, \ldots\}$ of individual variables, a countable set $\Cost = \{a_0, a_1, a_2, \ldots\}$ of individual constants, countable sets $\Pre = \{P^n_i\}$ of predicate constants and $\Fun = \{f^n_i\}$ of function letters, the operators $\lnot, \land, \lor, \rightarrow$, the quantifiers $\forall, \exists$, and parentheses. Terms and well-formed formulae are generated as usual; $\Ter$ is the set of terms, $\Ter^0$ the set of closed terms, $\Frml$ the set of formulae. We reserve $x,y,z$ for variables, $a,b,c$ for constants, $P,Q,R$ for predicates, and $\varphi,\psi,\chi$ as metavariables for formulae. A quantifier binds as far to the right as the parentheses allow, so that $\forall x\,\varphi \rightarrow \psi$ reads $(\forall x\,\varphi)\rightarrow\psi$. Free and bound occurrences, sentences and open formulae are defined in the standard way; we refer to \cite{VanDalen2013} for the details we do not repeat.

\begin{dfn}[Degree and weight]\label{def:grado}
The \emph{degree} $\dg{\varphi}$ of $\varphi$ is the number of occurrences in $\varphi$ of the operators $\lnot,\land,\lor,\rightarrow$ and of the quantifiers $\forall,\exists$. The \emph{weight} $\wt{\varphi}$ is
\[
\wt{\varphi} := 2k(\varphi) + l(\varphi),
\]
where $k(\varphi)$ counts the occurrences of $\land,\lor,\rightarrow,\forall,\exists$ in $\varphi$ and $l(\varphi)$ the occurrences of $\lnot$.
\end{dfn}

Substitution introduces no operator and removes none, so $\dg{\varphi[x/t]} = \dg{\varphi}$ and $\wt{\varphi[x/t]} = \wt{\varphi}$ for every term $t$. This small observation carries most of the inductions below: it is what allows an induction on the degree, or on the weight, to be applied to the \emph{instances} of proper subformulae, which is exactly what the quantifier rules produce. The two measures are not interchangeable, and each is used where it is the right one. The degree is the natural measure for an induction that follows the construction of a formula, and it is the one under which cut is eliminated in Section~\ref{sec:strutturale}. The weight is the measure under which the rules of the calculus are \emph{reductive}, in the sense made precise by Lemma~\ref{lem:peso}; the degree is not, since the rule $\lnot\lor$ carries $\lnot(\varphi\lor\psi)$ into the pair $\lnot\varphi, \lnot\psi$ without diminishing the sum of the degrees.

Substitution $\varphi[x/u]$ is defined by recursion on $\varphi$, distributing over the atomic and molecular cases and acting vacuously on $Qy\,\psi$ when $y \equiv x$. It is \emph{legitimate}, written $\Leg\,\varphi[x/u]$, when no occurrence of $u$ becomes bound after the replacement. Every substitution occurring in the rules of the calculus is on a closed term, and for closed terms legitimacy is automatic: a closed term contains no variable that a quantifier could capture.

\begin{con}[Relational fragment]\label{conv:relazionale}
Throughout the metatheory we assume $\Fun = \varnothing$, so that $\Ter^0 = \Cost$ and every instantiation takes place on a constant. Nothing essential is lost -- a language with functions reduces, preserving satisfiability, to a relational one through the graphs of the functors -- and the canonical model of Section~\ref{sec:metateoria} becomes free of complications. Remark~\ref{oss:herbrand} records what changes when $\Fun \neq \varnothing$.
\end{con}

We reserve, once and for all, the constant $a_0 \in \Cost$; Definition~\ref{def:parametri} explains its role, and no rule of the calculus is ever allowed to introduce it as a fresh constant.

\subsection{Semantics}

Rather than evaluate open formulae under assignments, we evaluate sentences in an expanded language, which is the arrangement that fits a calculus operating on closed formulae and constants.

\begin{dfn}[Model, expansion, valuation]\label{def:modello}
A \emph{model} for $\Ll$ is a pair $M = (D,i)$ with $D$ a non-empty set and $i$ an interpretation function: $i(a) \in D$ for $a \in \Cost$, $i(f^n_j) : D^n \to D$, $i(P^n_j) \subseteq D^n$. The class of models is $\Mod$. Given $M$, the \emph{expanded language} $\LD$ is obtained by adjoining to $\Ll$ a new constant $c_d$ for each $d \in D$, and $M$ is understood to interpret $c_d$ as $d$. The valuation $v = v_M : \Ter^0 \to D$ is defined by $v(a) = i(a)$, $v(c_d) = d$ and $v(f(t_1,\ldots,t_n)) = i(f)(v(t_1),\ldots,v(t_n))$.
\end{dfn}

\begin{dfn}[Truth]\label{def:verita}
For $\varphi$ a sentence of $\LD$, the relation $M \vDash \varphi$ is defined by recursion on $\dg{\varphi}$:
\begin{center}
\begin{tabular}{l @{\quad} c @{\quad} l}
$M \vDash P(t_1,\ldots,t_n)$ & iff & $(v(t_1),\ldots,v(t_n)) \in i(P)$;\\[2pt]
$M \vDash \lnot\varphi$ & iff & $M \nvDash \varphi$;\\[2pt]
$M \vDash \varphi \land \psi$ & iff & $M \vDash \varphi$ and $M \vDash \psi$;\\[2pt]
$M \vDash \varphi \lor \psi$ & iff & $M \vDash \varphi$ or $M \vDash \psi$;\\[2pt]
$M \vDash \varphi \rightarrow \psi$ & iff & $M \nvDash \varphi$ or $M \vDash \psi$;\\[2pt]
$M \vDash \forall x\,\varphi$ & iff & $M \vDash \varphi[x/c_d]$ for every $d \in D$;\\[2pt]
$M \vDash \exists x\,\varphi$ & iff & $M \vDash \varphi[x/c_d]$ for some $d \in D$.
\end{tabular}
\end{center}
\end{dfn}

The recursion is well founded: $\dg{\varphi[x/c_d]} = \dg{\varphi} < \dg{Qx\,\varphi}$. Since every element of $D$ has a name in $\LD$, no arbitrary choice of an auxiliary constant has to be made, and the two invariance lemmas that such a choice would require reduce to the following two statements, both immediate.

\begin{lem}[Coincidence]\label{lem:coincidenza}
If two models $M, M'$ have the same domain and agree on the interpretation of every constant, function letter and predicate occurring in the sentence $\varphi$ of $\LD$, then $M \vDash \varphi$ iff $M' \vDash \varphi$.
\end{lem}
\begin{proof}
Induction on $\dg{\varphi}$. In the atomic case the terms of $\varphi$ receive the same values on the two sides, by induction on their structure, and the predicate is interpreted alike. The molecular cases follow from the induction hypothesis applied to the immediate subformulae. For $\forall x\,\psi$: the sentences $\psi[x/c_d]$ have degree $\dg{\psi} < \dg{\forall x \psi}$, and their symbols are those of $\psi$ together with $c_d$, which the two models interpret alike by hypothesis on the common domain; so the induction hypothesis applies for each $d$, and the universal clause transfers. The case $\exists x\,\psi$ is the same with ``for some''.
\end{proof}

\begin{lem}[Semantic substitution]\label{lem:sostituzione}
Let $\varphi$ have at most $x$ free and let $t$ be a closed term of $\LD$. Then $M \vDash \varphi[x/t]$ if and only if $M \vDash \varphi[x/c_{v(t)}]$.
\end{lem}
\begin{proof}
Induction on $\dg{\varphi}$. If $\varphi$ is atomic, the two sentences differ only in that $t$ occupies the positions occupied by $c_{v(t)}$, and both terms denote $v(t)$; an induction on the structure of the terms of $\varphi$ shows that corresponding terms receive equal values, whence the two sides have the same truth value. The molecular cases are immediate, substitution distributing over the immediate subformulae. If $\varphi = Qy\,\psi$ with $y \equiv x$ the substitution is vacuous on both sides. If $y \not\equiv x$, then $\varphi[x/t] = Qy\,(\psi[x/t])$, and for each $d \in D$ the substitutions on the distinct variables $x,y$ commute, $t$ and $c_d$ being closed; the induction hypothesis applied to $\psi[y/c_d]$, of degree $\dg{\psi} < \dg{\varphi}$, gives $M \vDash \psi[y/c_d][x/t]$ iff $M \vDash \psi[y/c_d][x/c_{v(t)}]$, and quantification over $d$ yields the claim.
\end{proof}

Logical truth $\Vdash_{\Mod} \varphi$, logical consequence $\Gamma \Vdash_{\Mod} \varphi$ and satisfiability $\Sat_{\Mod}(\Gamma)$ are defined as usual for sets of sentences of $\Ll$. As in the propositional case, $\Gamma \Vdash_{\Mod} \varphi$ holds if and only if $\Gamma \cup \{\lnot\varphi\}$ is unsatisfiable; the whole tableau method turns on this hinge.

\subsection{Sequents and blocks}\label{sec:sequenti}

A \emph{sequent} is an expression $\Gamma \segue \Delta$ with $\Gamma, \Delta$ finite, possibly empty, sets of sentences of $\Ll$; the arrow is metalinguistic. The sequent is \emph{valid} when $\bigwedge\Gamma \Vdash_{\Mod} \bigvee\Delta$. Since $M \nvDash \bigvee\Delta$ iff $M \vDash \lnot\psi$ for every $\psi \in \Delta$, validity is equivalent to the unsatisfiability of the union of the antecedent with the negated succedents. We therefore set
\[
\Pi := \Gamma \cup \lnot[\Delta] = \{\varphi_1,\ldots,\varphi_n, \lnot\psi_1,\ldots,\lnot\psi_m\},
\qquad \lnot[\Delta] := \{\lnot\psi \mid \psi \in \Delta\},
\]
and call $\Pi$ the \emph{block associated} with $\Gamma \segue \Delta$; thus
\begin{equation}\label{eq:cardine}
\Gamma \segue \Delta \text{ is valid} \quad\text{iff}\quad \text{not } \Sat_{\Mod}(\Pi).
\end{equation}

\begin{rem}[The translation is not injective]\label{oss:non-iniettiva}
Distinct sequents may share a block: $\{\lnot\psi\}$ is the block of $\lnot\psi \segue$ and equally of $\segue \psi$. Nothing is lost, since two sequents with the same block are valid or invalid together by \eqref{eq:cardine}; but the collapse has to be kept in mind in Section~\ref{sec:LK}, where the passage from $\Pi \segue$ back to $\Gamma \segue \Delta$ requires a weakening step precisely when a formula of $\lnot[\Delta]$ happens to belong to $\Gamma$ as well.
\end{rem}

\section{The block calculus}\label{sec:regole}

The calculus operates on finite sets of formulae, decomposing them until a contradiction surfaces on every branch or until a branch resists decomposition. We describe the data first, then the rules, and last the discipline of application that makes the search complete.

\subsection{Blocks, branches, closure}

A \emph{tableau} for a block $\Pi$ is a rooted tree, finite or infinite, whose root is labelled by $\Pi$ and whose successor nodes are obtained by applying one of the rules below to a formula of the block of the parent node. Unlike the classical presentation of Smullyan, in which each node carries a single formula, here each node carries the entire current block: a branch does not \emph{accumulate} formulae as it descends, it \emph{transports} them, the block of the parent being replaced by one or two child blocks according as the rule is conjunctive or disjunctive.

A \emph{literal} is an atomic formula or the negation of an atomic formula.

\begin{dfn}[Closure]\label{def:chiusura}
A block $\Pi$ is \emph{closed} if there is an atomic formula $\theta$ with $\theta \in \Pi$ and $\lnot\theta \in \Pi$; the branch it terminates is marked $\times$ and no rule is applied to it. A tableau is \emph{closed} if every branch of it is closed. We write $\bd{h}\Pi$ to mean that there is a closed tableau for $\Pi$ of height at most $h$, and $\vdash \Pi$ to mean that $\bd{h}\Pi$ for some $h$.
\end{dfn}

Closure demands the syntactic identity of the argument terms: the calculus does not resort to unification, but leaves to the operator the choice of the constant with which to instantiate the universal formulae, so that the resulting atoms coincide literally with their own negations. Restricting closure to \emph{literals}, rather than admitting complementary pairs of arbitrary formulae, is what makes the invertibility of the rules height-preserving, and hence what makes the proof of Theorem~\ref{thm:taglio} possible; nothing is lost, since closure on arbitrary pairs is admissible (Lemma~\ref{lem:chiusura-gen}). The same choice is made, and for the same reason, on the side of sequent calculi. Gentzen's own $\mathrm{LK}$ admits basic sequents $\alpha \segue \alpha$ with $\alpha$ arbitrary, whereas the presentations designed to carry the elimination of cut restrict them to atoms; that the construction depends on the restriction is stressed by Franks \cite[Ch.~3]{Franks2026}, and the block calculus inherits the dependence intact.

\begin{prop}[Soundness of the closure criterion]\label{prop:chiusura}
A closed block is not satisfiable.
\end{prop}
\begin{proof}
If $\theta, \lnot\theta \in \Pi$, no $M$ satisfies both, since $M \vDash \lnot\theta$ iff $M \nvDash \theta$.
\end{proof}

\subsection{Propositional rules}\label{sec:alfabeta}

Following Smullyan's classification, the propositional rules divide into conjunctive rules, which produce a single child, and disjunctive rules, which bifurcate.

\medskip
\noindent -- \textit{Rules of type $\alpha$} (conjunctive)
\[
\frac{\Pi, \lnot\lnot \varphi}{\Pi, \varphi} \; \text{$\lnot\lnot$}
\qquad
\frac{\Pi, \varphi \land \psi}{\Pi, \varphi, \psi} \; \text{$\land$}
\qquad
\frac{\Pi, \lnot(\varphi \lor \psi)}{\Pi, \lnot\varphi, \lnot\psi} \; \text{$\lnot\lor$}
\qquad
\frac{\Pi, \lnot(\varphi \rightarrow \psi)}{\Pi, \varphi, \lnot\psi} \; \text{$\lnot\rightarrow$}
\]

\medskip
\noindent -- \textit{Rules of type $\beta$} (disjunctive)
\[
\frac{\Pi, \lnot(\varphi \land \psi)}{\Pi, \lnot\varphi \mid \Pi, \lnot\psi} \; \text{$\lnot\land$}
\qquad
\frac{\Pi, \varphi \lor \psi}{\Pi, \varphi \mid \Pi, \psi} \; \text{$\lor$}
\qquad
\frac{\Pi, \varphi \rightarrow \psi}{\Pi, \lnot\varphi \mid \Pi, \psi} \; \text{$\rightarrow$}
\]

In each rule the formula displayed above the line -- the \emph{principal} formula -- is removed from the block and replaced by the formulae below the line, the rest of the block being inherited unchanged. In the $\alpha$ rules a single child receives both resulting formulae; in the $\beta$ rules two children receive respectively the one and the other.

\subsection{Quantifier rules}\label{sec:gammadelta}

A $\gamma$ formula asserts something about \emph{each} element of the domain, and to exploit it one must be able to instantiate it on an arbitrary constant, and more than once; a $\delta$ formula asserts the existence of \emph{one} element with a property, and to instantiate it is to name a constant for such an element, with the constraint that the name be fresh. We write $\adown$ for a constant already among the parameters of the branch and $\aup$ for one fresh for it.

\begin{dfn}[Parameters]\label{def:parametri}
Let $\mathcal R$ be a branch of a tableau. The set of \emph{parameters} of $\mathcal R$ is
\[
\mathrm{Par}(\mathcal R) := \{a \in \Cost \mid a \text{ occurs in some block of } \mathcal R\} \cup \{a_0\},
\]
where $a_0$ is the reserved constant of Section~\ref{sec:preliminari}. Two features of this definition matter later. It is \emph{monotone}: a branch that grows never loses a parameter, so that the notion is stable along infinite branches. And it is never empty, which reflects on the syntactic side the requirement that the domains of the semantics be non-empty; the $\gamma$ rules, as the preservation lemma will show, are sound for any choice of instantiating constant, so no harm comes of always having $a_0$ available.
\end{dfn}

\begin{dfn}[Fresh constant]\label{def:nuova}
Let a $\delta$ rule be applied to a block $\Pi$ on a branch $\mathcal R$. A constant $a$ is \emph{fresh for $\mathcal R$}, written $\aup$, if $a \neq a_0$, if $a \notin \mathrm{Par}(\mathcal R)$, and if $a$ has not already been used as the fresh constant of an earlier $\delta$ application on $\mathcal R$.
\end{dfn}

The last clause is not redundant, although it looks so. When $x$ does not occur free in $\varphi$, the instance $\varphi[x/\aup]$ is $\varphi$ itself, and the constant introduced by the rule occurs in no block of the branch; without the clause it would be available again, and a later $\delta$ application could reuse it. It is the only case in which a $\delta$ application fails to enlarge the set of parameters, and Proposition~\ref{prop:crescita} is stated with it in view.

\medskip
\noindent -- \textit{Rules of type $\gamma$} (universal; $\adown \in \mathrm{Par}(\mathcal R)$)
\[
\frac{\Pi, \forall x\,\varphi}{\Pi, \forall x\,\varphi, \; \varphi[x/\adown]} \; \text{$\forall$}
\qquad
\frac{\Pi, \lnot \exists x\,\varphi}{\Pi, \lnot \exists x\,\varphi, \; \lnot\varphi[x/\adown]} \; \text{$\lnot\exists$}
\]

\medskip
\noindent -- \textit{Rules of type $\delta$} (existential; $\aup$ fresh for $\mathcal R$)
\[
\frac{\Pi, \exists x\,\varphi}{\Pi, \varphi[x/\aup]} \; \text{$\exists$}
\qquad
\frac{\Pi, \lnot \forall x\,\varphi}{\Pi, \lnot\varphi[x/\aup]} \; \text{$\lnot\forall$}
\]

Two features separate the families. The principal formula of a $\gamma$ rule is \emph{not consumed}: it survives in the child block alongside its own instance, because it may have to be instantiated again on a different constant later in the construction. The principal formula of a $\delta$ rule is consumed, the fresh constant discharging once and for all the task of representing a witness. And whereas in the $\alpha$ and $\beta$ rules the choice of which formula to decompose is immaterial for soundness, here the order matters for the success of the search: if a branch is to close on a given constant, the $\delta$ rule that introduces it must precede every $\gamma$ instantiation that will draw on it, with $\adown = \aup$, so that the resulting atoms coincide literally.

\begin{rem}[Two readings of freshness]\label{oss:freschezza}
The condition of Definition~\ref{def:nuova} is syntactic: $\aup$ does not occur on the branch. In the proof of completeness it is this reading that operates, since it guarantees that each $\delta$ application introduces a distinct witness. In the proof of soundness the semantic counterpart intervenes: one extends the interpretation by setting $\aup$ equal to a witness $d \in D$, and the coincidence lemma ensures that the extension does not alter the truth value of the context, in which $\aup$ does not occur. A new name, and a witness free to be interpreted: two faces of one condition, and it is useful to keep them distinct.
\end{rem}

\begin{rem}[Necessity of the freshness condition]\label{oss:novita}
The restriction on the $\delta$ rules separates a sound calculus from an unsound one. Take the block
\[
\Pi = \{\exists x\,P(x),\ \lnot\forall x\,P(x)\}
\]
and apply $\exists$ with $\aup = b$, obtaining $P(b)$. Were one then to apply $\lnot\forall$ reusing $b$, one would obtain $\lnot P(b)$, and the block would close. But $\Pi$ is satisfiable, as a domain with two elements only one of which satisfies $P$ shows, so the closure would be spurious. The obligation to choose a second constant $c \neq b$ leaves the block $\{P(b), \lnot P(c)\}$ open, as it must be; see Example~\ref{es:aperto}.
\end{rem}

We record here, once and for all, the sense in which the rules diminish what they act on.

\begin{lem}[Reduction of the weight]\label{lem:peso}
In every rule of the calculus, each formula introduced below the line -- excluding the principal formula that a $\gamma$ rule retains -- has weight strictly smaller than the principal formula. Moreover, for the rules $\alpha$, $\beta$ and $\delta$ the sum of the weights of the formulae of each child block is strictly smaller than the corresponding sum for the parent block.
\end{lem}
\begin{proof}
Both claims are read off a computation of the weight on the principal formulae, the context being common to parent and children. Directly from Definition~\ref{def:grado},
\[
\wt{\lnot\lnot\varphi} = \wt{\varphi}+2, \qquad
\wt{\varphi\star\psi} = \wt{\varphi}+\wt{\psi}+2,
\]
\[
\wt{\lnot(\varphi\star\psi)} = \wt{\varphi}+\wt{\psi}+3,
\]
for $\star \in \{\land,\lor,\rightarrow\}$, and, for $Q \in \{\forall,\exists\}$,
\[
\wt{Qx\,\varphi} = \wt{\varphi}+2, \qquad \wt{\lnot Qx\,\varphi} = \wt{\varphi}+3,
\]
while an instance $\varphi[x/t]$ has the weight of $\varphi$.

The first claim follows: the rule $\rightarrow$, to take the tightest case, introduces $\lnot\varphi$ of weight $\wt{\varphi}+1$ and $\psi$ of weight $\wt{\psi}$, both strictly below $\wt{\varphi}+\wt{\psi}+2$. The second is the same computation summed: the $\alpha$ rules diminish the total by $2$, except $\lnot\lor$, which diminishes it by $1$; each $\beta$ rule replaces the principal formula by a single formula of strictly smaller weight; and each $\delta$ rule diminishes the total by $2$. The exclusion of the $\gamma$ rules is not an oversight: their child block contains the principal formula together with a new instance, so the total weight \emph{increases}, by $\wt{\varphi}$ in the case of $\forall$ and by $\wt{\varphi}+1$ in the case of $\lnot\exists$; it is precisely this failure of reductivity that makes the first-order construction liable not to terminate.
\end{proof}

\subsection{Termination of applications, and the fair strategy}

A rule may be applied to a block without changing it, and a branch on which only such applications remain has nothing further to yield. The following notion separates the two situations.

\begin{dfn}[Productive application, terminated block]\label{def:produttiva}
An application of a rule is \emph{productive} if at least one of the child blocks differs from the parent. The rules $\alpha, \beta, \delta$ remove their principal formula, so every application of them is productive; an application of a $\gamma$ rule is unproductive precisely when the instance it would introduce already belongs to the block. A block is \emph{terminated} when it is closed, or when no productive application to it exists, the $\gamma$ applications considered being those on the parameters of the branch; in the second case it is \emph{completed open}, and marked $\odot$.
\end{dfn}

In the propositional case, by Lemma~\ref{lem:peso}, every rule strictly decreases the total weight of the block, and the tableau is finite. In the predicate case the $\gamma$ rule is not reductive: the universal formula survives alongside its instance, and the block may grow. The construction may therefore fail to terminate, and by the undecidability of first-order validity \cite{Church1936,Turing1936} no procedure can decide in every case, in finite time, whether a block is closable. What holds -- and this is what completeness establishes -- is that an unsatisfiable block is refuted by every tableau built according to a fair strategy.

\begin{dfn}[Fair strategy]\label{def:equa}
A construction is \emph{fair} if for every branch $\mathcal R$ that is never closed:
\begin{enumerate}
\item[\textup{(F1)}] every formula of type $\alpha$, $\beta$ or $\delta$ occurring in a block of $\mathcal R$ is the principal formula of an application at some finite stage of $\mathcal R$;
\item[\textup{(F2)}] for every formula of type $\gamma$ occurring in a block of $\mathcal R$ and every $a \in \mathrm{Par}(\mathcal R)$, the corresponding instance on $a$ is introduced at some finite stage of $\mathcal R$.
\end{enumerate}
\end{dfn}

Condition (F2) is infinitary: each $\delta$ application may enlarge the set of parameters, and every $\gamma$ formula has to be instantiated on each of them. That such a strategy exists is therefore not immediate, and the obligations to be discharged cannot be enumerated in advance, since they are created as the tree grows. The following construction organizes them explicitly.

\begin{prop}[Existence of a fair strategy]\label{prop:esistenza-equa}
There is an effective procedure that, applied to any initial block $\Pi$, constructs a fair tableau for $\Pi$.
\end{prop}
\begin{proof}
Call an \emph{obligation} on a branch $\mathcal R$ either an occurrence of a formula of type $\alpha$, $\beta$ or $\delta$ in a block of $\mathcal R$, or a pair $\langle \psi, a \rangle$ with $\psi$ of type $\gamma$ occurring in a block of $\mathcal R$ and $a \in \mathrm{Par}(\mathcal R)$. Each open branch carries a queue of pending obligations, processed first in, first out; at a bifurcation the queue is copied into both children.

The procedure runs in stages. At stage $0$ the queue of the single branch is initialized with the obligations generated by $\Pi$, in a fixed order determined by an effective enumeration of formulae and constants. At stage $n+1$, each open branch $\mathcal R$ removes the obligation at the head of its queue and discharges it: a formula of type $\alpha$, $\beta$ or $\delta$ is decomposed by the corresponding rule, and a pair $\langle \psi, a \rangle$ produces the $\gamma$ instance of $\psi$ on $a$, if that instance is not already present. Whatever new obligations this step creates -- the components of the decomposed formula, and, when a $\delta$ application has enlarged $\mathrm{Par}(\mathcal R)$ by a constant $b$, the pairs $\langle \psi, b\rangle$ for every $\gamma$ formula $\psi$ currently on $\mathcal R$ -- are appended to the tail of the queue of each branch on which they arise, again in the fixed order. Only finitely many are appended at each stage, since blocks and branches are finite at every finite stage. No pair $\langle \psi, a\rangle$ is appended twice on the same branch: a formula of type $\gamma$ persists once it has appeared, being never the principal formula of a rule that removes it, and $\mathrm{Par}(\mathcal R)$ only grows, so the pair has a first stage at which both of its components are present, and it is appended there and nowhere else.

An obligation appended to a queue at stage $n$ is preceded in that queue by finitely many others, and one obligation is removed from each open branch at each stage; hence it reaches the head, and is discharged, after finitely many stages. This gives (F1) at once. It gives (F2) as well: if $\psi$ is of type $\gamma$ on $\mathcal R$ and $a \in \mathrm{Par}(\mathcal R)$, then either $a$ is a parameter already when $\psi$ appears, in which case the pair $\langle\psi,a\rangle$ is appended then, or $a$ enters $\mathrm{Par}(\mathcal R)$ at a later stage, in which case the pair is appended at that stage; in both cases it is discharged after finitely many further stages, and the $\gamma$ instance of $\psi$ on $a$ is introduced. Note that the reserved parameter $a_0$ belongs to $\mathrm{Par}(\mathcal R)$ from the outset, so its obligations are among those initialized at stage $0$.
\end{proof}

When closure is expected within a few steps one adopts, in practice, a priority strategy: the $\alpha$ and $\delta$ nodes first, then the $\gamma$ nodes, and last the $\beta$ nodes \cite{Palladino2002}. It is a particular case of Definition~\ref{def:equa}, oriented towards efficiency.

\subsection{Absorption of the structural rules}

The calculus contains no explicit structural rule. This does not mean that the structural operations are absent: they are absorbed into the data structure and into the closure criterion.

\begin{prop}[Structural absorption]\label{prop:strutturali}
In the block calculus, contraction is absorbed by the set-theoretic nature of the blocks, which record no multiplicities; weakening is absorbed by the closure criterion, since a block closes as soon as it contains a complementary pair of literals, independently of what else it contains; and left contraction on universal formulae is realized by the persistence of the principal formula in the $\gamma$ rules.
\end{prop}
\begin{proof}
The first two claims are immediate from Definitions~\ref{def:chiusura} and \ref{def:parametri} and from $\Pi \cup \{\theta,\theta\} = \Pi \cup \{\theta\}$. For the third: in $\mathrm{LK}$ the rule $\forall\mathrm{s}$ consumes $\forall x\,\alpha$, so that instantiating it on distinct constants requires duplicating it beforehand, an operation made available by explicit contraction or by the set-theoretic reading of antecedents. In the block calculus the $\gamma$ rule retains $\forall x\,\varphi$ in the child, making available without duplication all the instantiations that fairness requires. The combinatorial role is the same.
\end{proof}

Proposition~\ref{prop:strutturali} states that the structural operations are not \emph{needed} as rules. It does not yet state that they are \emph{admissible} as transformations of closed tableaux, which is a different and stronger assertion; that is the business of the next section.

\section{Structural analysis}\label{sec:strutturale}

This section is proof-theoretic in the narrow sense: no semantic notion occurs in it. We show that the operations one would want as structural rules are admissible transformations on closed tableaux, and that all of them preserve the height; on that basis we prove that cut is admissible. The order of the results is the order of their dependence, and it follows the pattern familiar from the $\mathrm{G3}$ calculi \cite{NegriVonPlato2001}: generalized closure, substitution, weakening, invertibility, cut.

We record first a fact used repeatedly.

\begin{lem}[Finiteness]\label{lem:finitezza}
A closed tableau is a finite tree.
\end{lem}
\begin{proof}
Every branch of a closed tableau ends, after finitely many steps, in a closed leaf. An infinite tree in which every node has at most two children would possess, by K\"onig's lemma \cite{Konig1927}, an infinite branch, and an infinite branch ends in no leaf.
\end{proof}

A second observation, equally elementary, will be used without further mention: \emph{no rule of the calculus has a literal as its principal formula}. Every principal formula is either a conjunction, a disjunction, an implication, a quantified formula, or the negation of one of these; a literal is an atom or the negation of an atom. This is what makes the base cases of the inductions below go through.

\subsection{Generalized closure}

The closure criterion of Definition~\ref{def:chiusura} speaks of literals only. Nothing is lost.

\begin{lem}[Generalized closure]\label{lem:chiusura-gen}
For every block $\Pi$ and every formula $\theta$, $\; \bd{2\dg{\theta}} \Pi \cup \{\theta, \lnot\theta\}$.
\end{lem}
\begin{proof}
By induction on $\dg{\theta}$. If $\theta$ is atomic the block is closed by Definition~\ref{def:chiusura}, and the height is $0$.

If $\theta = \lnot\chi$, apply $(\lnot\lnot)$ to $\lnot\lnot\chi$: the child is $\Pi \cup \{\lnot\chi, \chi\}$, which by the induction hypothesis admits a closed tableau of height at most $2\dg{\chi}$; the total height is at most $1 + 2\dg{\chi} \le 2\dg{\theta}$.

If $\theta = \chi \land \xi$, apply $(\land)$ and then $(\lnot\land)$: the two children are $\Pi \cup \{\chi,\xi,\lnot\chi\}$ and $\Pi \cup \{\chi,\xi,\lnot\xi\}$, closed by the induction hypothesis at heights at most $2\dg{\chi}$ and $2\dg{\xi}$; two steps have been used, and $2 + \max(2\dg{\chi},2\dg{\xi}) \le 2\dg{\theta}$. The cases $\theta = \chi\lor\xi$ and $\theta = \chi\rightarrow\xi$ are the same computation with $(\lor)$ then $(\lnot\lor)$, respectively $(\rightarrow)$ then $(\lnot\rightarrow)$, and with the roles of the two branches exchanged.

If $\theta = \forall x\,\chi$, apply $(\lnot\forall)$ to $\lnot\forall x\,\chi$ with a fresh constant $b$, obtaining $\Pi \cup \{\forall x\,\chi, \lnot\chi[x/b]\}$, and then $(\forall)$ with $\adown := b$ when $x$ occurs free in $\chi$, so that $b$ is a parameter of the branch, and with $\adown := a_0$ otherwise, the substitution being vacuous in that case. Either way the child is $\Pi' \cup \{\chi[x/c], \lnot\chi[x/c]\}$ for the constant $c$ used, and the induction hypothesis applies, $\dg{\chi[x/c]} = \dg{\chi} < \dg{\theta}$; the height is at most $2 + 2\dg{\chi} = 2\dg{\theta}$. The case $\theta = \exists x\,\chi$ is dual, with $(\exists)$ followed by $(\lnot\exists)$.
\end{proof}

\subsection{Substitution and weakening}

The transformations of this subsection all rest on the freedom to rename the constants that the $\delta$ rules introduce. We isolate that freedom first, in the form in which the freshness condition makes it available, and then derive from it the admissibility of substitution and of weakening.

\begin{lem}[Renaming of a witness]\label{lem:rinomina-tab}
Let $T$ be a closed tableau for $\Pi$ of height at most $h$, let $a$ be the constant introduced by some $\delta$ application of $T$, and let $b \neq a_0$ be a constant occurring nowhere in $T$. Replacing $a$ by $b$ in the instance introduced by that application and in every block of the subtree below it yields a closed tableau for $\Pi$ of height at most $h$.
\end{lem}
\begin{proof}
The root is untouched: $a$ was fresh when introduced, hence occurs in no block above the application, and in particular not in $\Pi$. Below the application the replacement is uniform, so every rule application remains an application of the same rule -- instances are carried to instances -- and every complementary pair of literals $\theta, \lnot\theta$ is carried to the complementary pair $\theta[a/b], \lnot\theta[a/b]$, so that all leaves remain closed. The side conditions survive: the $\delta$ application in question now introduces $b$, which occurs nowhere above it because it occurred nowhere in $T$, and $b \neq a_0$; a later $\delta$ application introducing a constant $c$ still introduces one foreign to its renamed branch, $c$ being distinct from $a$, which was already on the branch, and from $b$, which occurred nowhere; and a $\gamma$ application instantiates on a constant that is still a parameter of its renamed branch, occurrences having been renamed uniformly along it. No step is added or removed, so the height does not grow.
\end{proof}

\begin{lem}[Height-preserving admissibility of substitution]\label{lem:sost}
Let $a \neq a_0$ and let $b$ be any constant. If $\bd{h} \Pi$ then $\bd{h} \Pi[a/b]$.
\end{lem}
\begin{proof}
Let $T$ be a closed tableau for $\Pi$ of height at most $h$; by Lemma~\ref{lem:finitezza} it is finite, and performs finitely many $\delta$ applications. Applying Lemma~\ref{lem:rinomina-tab} to each of them in turn, with replacement constants chosen pairwise distinct, distinct from $a_0$, from $a$ and from $b$, and occurring nowhere in $T$ -- such constants exist, $\Cost$ being countable and $T$ finite -- we may assume that no constant introduced as fresh in $T$ is $a$ or $b$. Now replace $a$ by $b$ uniformly throughout $T$.

Complementary pairs of literals go to complementary pairs of literals, so every leaf remains closed. The rules $\alpha$ and $\beta$ commute with the replacement, substitution of a constant for a constant distributing over the formation of components, and they carry no side condition. Consider a $\gamma$ application on the constant $c$: after the replacement it is an application on $c[a/b]$, and this is a parameter of the renamed branch. Indeed, if $c \neq a$ then $c$ is unaffected, and it either occurred in a block of the branch -- an occurrence that the replacement does not destroy -- or it is $a_0$, which belongs to every set of parameters; while if $c = a$, then $a$ occurred in some block of the branch, since $a \neq a_0$, and $b$ occurs in the corresponding renamed block. Consider finally a $\delta$ application with witness $w$: by the preliminary renaming $w \notin \{a,b\}$, so the occurrences of $w$ are exactly those it had before, and it remains fresh for its branch, distinct from $a_0$ and from every earlier witness. The transformation adds no step.
\end{proof}

\begin{lem}[Height-preserving admissibility of weakening]\label{lem:wk}
If $\bd{h}\Pi$ then $\bd{h}\Pi \cup \Sigma$ for every finite set $\Sigma$ of sentences.
\end{lem}
\begin{proof}
Let $T$ be a closed tableau for $\Pi$ of height at most $h$, finite by Lemma~\ref{lem:finitezza}. By Lemma~\ref{lem:rinomina-tab}, applied to each of its finitely many $\delta$ applications, we may assume that no constant introduced as fresh in $T$ occurs in $\Sigma$. Let $T'$ be obtained from $T$ by adjoining $\Sigma$ to every block. Every rule application of $T$ remains one in $T'$: the rules act on a principal formula and transmit an arbitrary context, which has merely been enlarged; every $\gamma$ instantiation is still on a parameter of the enlarged branch; and every $\delta$ application still introduces a constant foreign to its branch, the constants of $\Sigma$ having been avoided. Every leaf remains closed, its complementary pair being undisturbed, and no step has been added.
\end{proof}

\subsection{Invertibility}

Invertibility is here the assertion that what a rule licenses downwards it also licenses upwards: if the block to which a rule would apply is refutable, then so is each block the rule would produce. In the presence of closure on literals it holds with preservation of the height, and this is what makes it usable in the proof of cut admissibility.

\begin{lem}[Height-preserving invertibility]\label{lem:inv}
Let $\Pi$ be a block, $h \ge 0$.
\begin{enumerate}
\item[\textup{(i)}] If $\bd{h}\Pi, \lnot\lnot\varphi$ then $\bd{h}\Pi, \varphi$. If $\bd{h}\Pi, \varphi\land\psi$ then $\bd{h}\Pi, \varphi, \psi$. If $\bd{h}\Pi, \lnot(\varphi\lor\psi)$ then $\bd{h}\Pi, \lnot\varphi, \lnot\psi$. If $\bd{h}\Pi, \lnot(\varphi\rightarrow\psi)$ then $\bd{h}\Pi, \varphi, \lnot\psi$.
\item[\textup{(ii)}] If $\bd{h}\Pi, \varphi\lor\psi$ then $\bd{h}\Pi, \varphi$ and $\bd{h}\Pi, \psi$; and correspondingly for $\lnot(\varphi\land\psi)$ with components $\lnot\varphi, \lnot\psi$ and for $\varphi\rightarrow\psi$ with components $\lnot\varphi, \psi$.
\item[\textup{(iii)}] If $\bd{h}\Pi, \forall x\,\varphi$ then $\bd{h}\Pi, \forall x\,\varphi, \varphi[x/a]$ for every constant $a$; and correspondingly for $\lnot\exists x\,\varphi$ with $\lnot\varphi[x/a]$.
\item[\textup{(iv)}] If $\bd{h}\Pi, \exists x\,\varphi$ then $\bd{h}\Pi, \varphi[x/a]$ for every constant $a$; and correspondingly for $\lnot\forall x\,\varphi$ with $\lnot\varphi[x/a]$.
\end{enumerate}
\end{lem}
\begin{proof}
Statement (iii) is Lemma~\ref{lem:wk}, the block on the right containing the block on the left.

For (i), (ii) and (iv) the argument is one induction on $h$, and we display it for the representative cases; the remaining ones differ only in the name of the rule.

\emph{Case} (i), $\varphi\land\psi$. If $h = 0$ the block $\Pi \cup \{\varphi\land\psi\}$ is closed on a complementary pair of literals; a conjunction is not a literal, so the pair lies in $\Pi$, and $\Pi \cup \{\varphi,\psi\}$ is closed as well. Let $h > 0$ and let $r$ be the rule applied at the root, with principal formula $\chi$. If $\chi = \varphi\land\psi$, the child is $\Pi \cup \{\varphi,\psi\}$ and its subtableau has height at most $h-1$. If $\chi \neq \varphi\land\psi$, then $\chi \in \Pi$; the children of the root are $\Pi_j \cup \{\varphi\land\psi\}$ with $\bd{h-1}\Pi_j \cup \{\varphi\land\psi\}$, the induction hypothesis gives $\bd{h-1}\Pi_j \cup \{\varphi,\psi\}$, and $r$ may be re-applied. Its side conditions are met: a $\gamma$ instantiation was on a parameter of the old branch, and the constants of $\varphi\land\psi$ are those of $\varphi$ and $\psi$ together, so no parameter is lost; a $\delta$ application introduced a constant foreign to the old branch, and the new branch contains no constant that the old one did not contain.

\emph{Case} (ii), $\varphi\lor\psi$. For $h = 0$, a disjunction is not a literal, so the complementary pair lies in $\Pi$ and both $\Pi \cup \{\varphi\}$ and $\Pi\cup\{\psi\}$ are closed. For $h>0$: if the principal formula of the root rule is $\varphi\lor\psi$, the two children are exactly $\Pi\cup\{\varphi\}$ and $\Pi\cup\{\psi\}$, of height at most $h-1$; otherwise one applies the induction hypothesis to the children and re-applies the rule, as before.

\emph{Case} (iv), $\exists x\,\varphi$. Fix the constant $a$. Let $T$ be a closed tableau for $\Pi \cup \{\exists x\,\varphi\}$ of height at most $h$; by Lemma~\ref{lem:rinomina-tab} we may assume that no constant introduced as fresh in $T$ is $a$. We argue by induction on $h$. If $h = 0$, an existential formula is not a literal, so the complementary pair lies in $\Pi$ and $\Pi \cup \{\varphi[x/a]\}$ is closed. Let $h>0$, with root rule $r$ and principal formula $\chi$. If $\chi = \exists x\,\varphi$, the rule is $(\exists)$ with some fresh witness $b$, and $\bd{h-1}\Pi \cup \{\varphi[x/b]\}$; since $b \neq a_0$, Lemma~\ref{lem:sost} yields $\bd{h-1} (\Pi \cup \{\varphi[x/b]\})[b/a]$, and $b$ does not occur in $\Pi$, so this block is $\Pi \cup \{\varphi[x/a]\}$. If $\chi \neq \exists x\,\varphi$, then $\chi \in \Pi$ and the children of the root are $\Pi_j \cup \{\exists x\,\varphi\}$; the induction hypothesis gives $\bd{h-1}\Pi_j \cup \{\varphi[x/a]\}$, and $r$ is re-applied. Its side conditions are met: the constants of $\exists x\,\varphi$ occur in $\varphi[x/a]$ as well, so a $\gamma$ instantiation retains its parameter; and a $\delta$ witness of $T$ is distinct from $a$ by the preliminary renaming, hence still foreign to the new branch.
\end{proof}

\subsection{Admissibility of cut}

We come to the result that gives the previous lemmas their point. The rule of cut is not among the rules of the calculus; the question is whether adding it would let one refute blocks that are otherwise irrefutable, and the answer is that it would not.

\begin{thm}[Cut; Smullyan]\label{thm:taglio}
Let $\theta$ be a sentence. If $\vdash \Pi \cup \{\theta\}$ and $\vdash \Pi \cup \{\lnot\theta\}$, then $\vdash \Pi$.
\end{thm}

The statement is the eliminability of $\theta$ in the sense of \cite[Ch.~XII]{Smullyan1968}, and it is proved there in an abstract form covering block tableaux among other systems. The proof below is the specialization of that argument to the present notation, with one difference of arrangement: Smullyan's hypothesis $P_2$ is a property of a \emph{given} closed tableau, from which the invertibility of the rules is recovered afterwards as a corollary of the theorem, whereas here invertibility is Lemma~\ref{lem:inv}, established beforehand and used as an ingredient.
\begin{proof}
Let $T_1$ be a closed tableau for $\Pi \cup \{\theta\}$ of height $h_1$ and $T_2$ a closed tableau for $\Pi \cup \{\lnot\theta\}$ of height $h_2$. The argument is a double induction: the principal induction is on the degree $\dg{\theta}$ of the cut formula, the secondary one on the sum $h_1 + h_2$. Three configurations arise, according as one of the two roots is closed or neither is.

\smallskip
\emph{Configuration 1: $\Pi \cup \{\theta\}$ is closed}, on the complementary pair $p, \lnot p$ with $p$ atomic. If both members lie in $\Pi$, then $\Pi$ is closed. If $\theta = p$, then $\lnot p \in \Pi$, so $\Pi \cup \{\lnot\theta\} = \Pi$ and $T_2$ is already a closed tableau for $\Pi$. If $\theta = \lnot p$, then $p \in \Pi$ and $\lnot\theta = \lnot\lnot p$; Lemma~\ref{lem:inv}(i) applied to $\bd{h_2}\Pi \cup \{\lnot\lnot p\}$ gives $\bd{h_2}\Pi \cup \{p\}$, which is $\Pi$. These three possibilities are exhaustive, since $\theta$ cannot be both $p$ and $\lnot p$.

\smallskip
\emph{Configuration 2: $\Pi \cup \{\lnot\theta\}$ is closed}, on $p, \lnot p$. Either both members lie in $\Pi$, and $\Pi$ is closed; or $\lnot\theta = \lnot p$, whence $\theta = p \in \Pi$ and $\Pi \cup \{\theta\} = \Pi$, so that $T_1$ serves. The remaining possibility, $\lnot\theta = p$, is excluded, a negation not being atomic.

\smallskip
\emph{Configuration 3: neither root is closed.} Let $r_1$ be the rule applied at the root of $T_1$, with principal formula $\varphi_1$, and $r_2$ the rule applied at the root of $T_2$, with principal formula $\varphi_2$.

\emph{Case 3a: $\varphi_1 \neq \theta$}, hence $\varphi_1 \in \Pi$. Let $\Pi_1, \ldots, \Pi_k$ ($k \le 2$) be the blocks that $r_1$ produces from $\Pi$ with principal formula $\varphi_1$; then the children of the root of $T_1$ are the blocks $\Pi_j \cup \{\theta\}$, and $\bd{h_1 - 1}\Pi_j \cup \{\theta\}$ for each $j$. Lemma~\ref{lem:inv}, applied to $\bd{h_2}\Pi \cup \{\lnot\theta\}$ with the rule $r_1$ and the principal formula $\varphi_1 \in \Pi$, gives $\bd{h_2}\Pi_j \cup \{\lnot\theta\}$; here it matters that clauses (iii) and (iv) hold for \emph{every} constant, so that the same instantiating constant or the same witness may be used on the two sides. The secondary induction hypothesis, the degree of $\theta$ being unchanged and the sum of the heights having decreased, yields $\vdash \Pi_j$ for each $j$.

It remains to recompose. If $r_1$ is of type $\alpha$, $\beta$ or $\delta$, apply it to $\Pi$ with principal formula $\varphi_1$: the side condition of a $\delta$ rule is met, since a constant fresh for $\Pi \cup \{\theta\}$ is a fortiori fresh for $\Pi$. If $r_1$ is of type $\gamma$, with instantiating constant $\adown$, the single block $\Pi_1$ is $\Pi \cup \{\varphi[x/\adown]\}$, the principal formula being retained and already present in $\Pi$; should $\adown$ fail to be a parameter of $\Pi$ -- which can happen only if $\adown$ occurred in $\theta$ alone, and then $\adown \neq a_0$ -- Lemma~\ref{lem:sost} transforms $\vdash \Pi_1$ into $\vdash \Pi_1[\adown/a_0] = \Pi \cup \{\varphi[x/a_0]\}$, and $a_0$ is a parameter of every branch. In either case a single application of $r_1$ to $\Pi$ has the available blocks as its children, and $\vdash \Pi$.

\emph{Case 3b: $\varphi_2 \neq \lnot\theta$.} Symmetric to 3a, with the roles of the two tableaux exchanged.

\emph{Case 3c: $\varphi_1 = \theta$ and $\varphi_2 = \lnot\theta$.} The cut formula is principal on both sides; in particular $\theta$ is not a literal. We distinguish according to its principal sign.

If $\theta = \lnot\chi$, then $\lnot\theta = \lnot\lnot\chi$ and $r_2$ can only be $(\lnot\lnot)$, so that $\vdash \Pi \cup \{\chi\}$. Together with $\vdash \Pi \cup \{\lnot\chi\}$, which is $T_1$, a cut on $\chi$ -- of degree $\dg{\chi} < \dg{\theta}$ -- gives $\vdash \Pi$ by the principal induction hypothesis. This single argument disposes of every $\theta$ whose principal sign is a negation.

If $\theta = \varphi \land \psi$, then $r_1$ is $(\land)$ and $r_2$ is $(\lnot\land)$, so that $\vdash \Pi \cup \{\varphi,\psi\}$, $\vdash \Pi \cup \{\lnot\varphi\}$ and $\vdash \Pi \cup \{\lnot\psi\}$. By Lemma~\ref{lem:wk}, $\vdash \Pi \cup \{\varphi, \lnot\psi\}$; a cut on $\psi$, of smaller degree, gives $\vdash \Pi \cup \{\varphi\}$, and a cut on $\varphi$, again of smaller degree, gives $\vdash \Pi$.

If $\theta = \varphi \lor \psi$, then $\vdash \Pi\cup\{\varphi\}$, $\vdash \Pi\cup\{\psi\}$ and $\vdash \Pi \cup \{\lnot\varphi,\lnot\psi\}$. Weakening the second to $\vdash \Pi \cup \{\lnot\varphi, \psi\}$ and cutting on $\psi$ gives $\vdash \Pi \cup \{\lnot\varphi\}$; a cut on $\varphi$ with the first gives $\vdash \Pi$. The case $\theta = \varphi\rightarrow\psi$ runs likewise: from $\vdash \Pi\cup\{\lnot\varphi\}$, $\vdash \Pi\cup\{\psi\}$ and $\vdash\Pi\cup\{\varphi,\lnot\psi\}$, weaken the second to $\vdash \Pi\cup\{\varphi,\psi\}$, cut on $\psi$ to obtain $\vdash\Pi\cup\{\varphi\}$, and cut on $\varphi$.

If $\theta = \forall x\,\varphi$, then $r_1$ is the $\gamma$ rule $(\forall)$ with some instantiating constant $a$, so that $\bd{h_1-1} \Pi \cup \{\forall x\,\varphi, \varphi[x/a]\}$, and $r_2$ is the $\delta$ rule $(\lnot\forall)$ with a fresh witness $b$, so that $\bd{h_2-1}\Pi \cup \{\lnot\varphi[x/b]\}$. Since $b \neq a_0$ and $b$ does not occur in $\Pi$, Lemma~\ref{lem:sost} turns the latter into
\[
\bd{h_2-1} \Pi \cup \{\lnot\varphi[x/a]\}.
\]
Weakening $T_2$ gives $\bd{h_2}\Pi \cup \{\varphi[x/a], \lnot\forall x\,\varphi\}$; a cut on $\forall x\,\varphi$ between this and $\bd{h_1-1}\Pi \cup \{\varphi[x/a], \forall x\,\varphi\}$ is licensed by the secondary induction hypothesis, the degree being unchanged and the sum of the heights having decreased, and yields $\vdash \Pi \cup \{\varphi[x/a]\}$. A cut on $\varphi[x/a]$, whose degree $\dg{\varphi}$ is smaller than $\dg{\forall x\,\varphi}$, then gives $\vdash \Pi$ by the principal induction hypothesis.

If $\theta = \exists x\,\varphi$, the argument is the mirror image: $r_1$ is $(\exists)$ with fresh witness $b$ and $r_2$ is $(\lnot\exists)$ with instantiating constant $a$; Lemma~\ref{lem:sost} turns $\bd{h_1-1}\Pi\cup\{\varphi[x/b]\}$ into $\bd{h_1-1}\Pi\cup\{\varphi[x/a]\}$, a cut on $\exists x\,\varphi$ at reduced height gives $\vdash \Pi \cup \{\lnot\varphi[x/a]\}$, and a cut on $\varphi[x/a]$ concludes.

The cases are exhaustive, the principal sign of a non-literal $\theta$ being one of $\land, \lor, \rightarrow, \lnot, \forall, \exists$, and in each of them the appeal to the induction hypothesis is licit. This completes the proof.
\end{proof}

The proof is syntactic and effective: it describes a procedure that, from two closed tableaux, builds a third. Cut admissibility could also be obtained semantically, from soundness and completeness -- if the two blocks are unsatisfiable then $\Pi$ is, since any model of $\Pi$ satisfies $\theta$ or $\lnot\theta$ -- but that route gives no construction, and it makes the result depend on K\"onig's lemma through Theorem~\ref{thm:completezza}. The syntactic proof is what will be used in Section~\ref{sec:LK} to obtain the Hauptsatz for $\mathrm{LK}$ by transfer.

\begin{cor}[Cut adds nothing]\label{cor:taglio-conservativo}
Let the block calculus be extended by the rule
\[
\frac{\Pi, \theta \qquad \Pi, \lnot\theta}{\Pi} \; \mathrm{Cut}.
\]
A block admits a closed tableau in the extended calculus if and only if it admits one in the original calculus.
\end{cor}
\begin{proof}
One direction is trivial. For the other, replace the topmost application of $\mathrm{Cut}$ by the construction of Theorem~\ref{thm:taglio} and iterate; the extended tableau being finite, the procedure terminates.
\end{proof}

\begin{cor}[Analytic cut suffices]\label{cor:taglio-analitico}
If a block $\Pi$ is refutable with the help of $\mathrm{Cut}$, it is refutable without it; consequently the cut formulae may be restricted, without loss, to generalized subformulae of $\Pi$ in the sense of Definition~\ref{def:sottoformula-gen}.
\end{cor}
\begin{proof}
Immediate from Corollary~\ref{cor:taglio-conservativo} and Theorem~\ref{thm:subformula}, whose proof is given in Section~\ref{sec:analiticita} and does not depend on the present section: the cut-free refutation whose existence is thereby guaranteed contains only generalized subformulae of $\Pi$, and a fortiori every formula on which one might wish to cut may be taken among them.
\end{proof}

Corollary~\ref{cor:taglio-analitico} places the calculus on the map drawn by D'Agostino and Mondadori \cite{DAgostinoMondadori1994}: the restriction of cut to subformulae of the conclusion, which they introduce in order to recover concision without sacrificing analyticity, is here not an addition but a consequence.

\section{Soundness, completeness, compactness}\label{sec:metateoria}

We return to the semantics of Section~\ref{sec:preliminari}, under Convention~\ref{conv:relazionale}.

\subsection{Soundness}

The heart of soundness is a preservation property: no rule can lead from a satisfiable block to a family of children all of which are unsatisfiable. We prove it by an inspection of the rules, and we prove at the same time the converse, which costs four lines and says that no rule invents satisfiability either.

\begin{lem}[Preservation of satisfiability]\label{lem:conservazione}
Let a rule be applied to the block $\Pi'$, with children $\Pi_1$ and possibly $\Pi_2$. Then $\Pi'$ is satisfiable if and only if at least one child is satisfiable.
\end{lem}
\begin{proof}
($\Rightarrow$) Let $M \vDash \Pi'$, let $\varphi_0 \in \Pi'$ be the principal formula, and put $\Pi = \Pi' \setminus \{\varphi_0\}$, so that $M \vDash \Pi$. The eleven rules are treated in four groups.

\emph{Rules $\alpha$.} If $\varphi_0 = \varphi\land\psi$, the conjunction clause gives $M \vDash \varphi$ and $M \vDash \psi$, so $M$ satisfies the single child $\Pi, \varphi, \psi$. If $\varphi_0 = \lnot\lnot\varphi$, the negation clause applied twice gives $M \vDash \varphi$. If $\varphi_0 = \lnot(\varphi\lor\psi)$, then $M \nvDash \varphi\lor\psi$, so $M \nvDash \varphi$ and $M \nvDash \psi$, that is $M \vDash \lnot\varphi$ and $M \vDash \lnot\psi$. If $\varphi_0 = \lnot(\varphi\rightarrow\psi)$, then $M \vDash \varphi$ and $M \nvDash \psi$. In each case the unique child is satisfied by $M$ itself.

\emph{Rules $\beta$.} If $\varphi_0 = \varphi\lor\psi$, then $M \vDash \varphi$ or $M \vDash \psi$, and accordingly $M$ satisfies the first or the second child. If $\varphi_0 = \lnot(\varphi\land\psi)$, then $M \nvDash \varphi$ or $M \nvDash \psi$. If $\varphi_0 = \varphi\rightarrow\psi$, then $M \vDash \lnot\varphi$ or $M \vDash \psi$.

\emph{Rules $\gamma$.} Let $\varphi_0 = \forall x\,\varphi$ and let $\adown$ be the constant on which one instantiates; no freshness is required of it, so the case covers the reserved parameter $a_0$ as well. By the universal clause $M \vDash \varphi[x/c_d]$ for every $d \in D$, in particular for $d = v(\adown)$, and Lemma~\ref{lem:sostituzione} converts this into $M \vDash \varphi[x/\adown]$. Since the principal formula is retained, $M$ satisfies the whole child $\Pi, \forall x\,\varphi, \varphi[x/\adown]$. If $\varphi_0 = \lnot\exists x\,\varphi$, the existential clause gives $M \nvDash \varphi[x/c_d]$ for every $d$, and the same step yields $M \vDash \lnot\varphi[x/\adown]$.

\emph{Rules $\delta$.} Let $\varphi_0 = \exists x\,\varphi$; by the existential clause there is $d \in D$ with $M \vDash \varphi[x/c_d]$. Let $\aup$ be the fresh constant of the rule and let $M'$ be the model that agrees with $M$ everywhere except that $i'(\aup) = d$. Then $v_{M'}(\aup) = d$, so Lemma~\ref{lem:sostituzione} gives $M' \vDash \varphi[x/\aup]$ if and only if $M' \vDash \varphi[x/c_d]$; and $\aup$ does not occur in $\varphi[x/c_d]$, so Lemma~\ref{lem:coincidenza} transfers $M \vDash \varphi[x/c_d]$ to $M'$. Moreover $M' \vDash \chi$ for every $\chi \in \Pi$, since $\aup$ occurs in no such $\chi$ by the freshness condition and Lemma~\ref{lem:coincidenza} applies. Hence $M' \vDash \Pi, \varphi[x/\aup]$. It is exactly freshness that guarantees that the reinterpretation of $\aup$ does not falsify the context; without it the step would fail, as Remark~\ref{oss:novita} shows. The rule $(\lnot\forall)$ is treated in the same way, starting from a $d$ with $M \nvDash \varphi[x/c_d]$.

($\Leftarrow$) Suppose some child is satisfied by $M$. For the rules $\alpha$ and $\beta$ the truth clauses give $M \vDash \varphi_0$ at once -- from $M \vDash \varphi$ and $M \vDash \psi$ one recovers $M \vDash \varphi\land\psi$, from $M \vDash \varphi$ one recovers $M \vDash \varphi\lor\psi$, and so on -- and the context is common. For a $\gamma$ rule the child contains the parent, so nothing is to be shown. For a $\delta$ rule, from $M \vDash \varphi[x/\aup]$ and Lemma~\ref{lem:sostituzione} we get $M \vDash \varphi[x/c_{v(\aup)}]$, whence $M \vDash \exists x\,\varphi$ by the existential clause; and the rest of the block is untouched.
\end{proof}

\begin{thm}[Soundness]\label{thm:correttezza}
If there is a closed tableau for $\Pi$, then $\Pi$ is not satisfiable. Consequently, if the tableau for the block associated with $\Gamma \segue \Delta$ closes, the sequent is valid.
\end{thm}
\begin{proof}
Suppose $\Pi$ satisfiable and let $T$ be a closed tableau for it. We construct by recursion a descending sequence of nodes $\Pi = \Pi_0, \Pi_1, \ldots$ of $T$, each a child of the preceding, along which every block is satisfiable: the root is satisfiable by hypothesis, and if $\Pi_k$ is satisfiable and not a leaf, the rule applied to it has, by Lemma~\ref{lem:conservazione}, a satisfiable child, which we take as $\Pi_{k+1}$. By Lemma~\ref{lem:finitezza} the tableau is finite, so the sequence reaches a leaf $\Pi_n$ that is satisfiable. But every leaf of a closed tableau is a closed block, and a closed block is unsatisfiable by Proposition~\ref{prop:chiusura}. The claim about sequents follows from \eqref{eq:cardine}.
\end{proof}

\subsection{Persistence}

Before turning to completeness we isolate a monotonicity property of branches, the exact counterpart of the fact that no rule destroys a literal.

\begin{lem}[Persistence]\label{lem:persistenza}
Let $\mathcal R = \Pi_0, \Pi_1, \ldots$ be a branch of a tableau.
\begin{enumerate}
\item[\textup{(i)}] If $\psi \in \Pi_k$ and $\psi$ is not the principal formula of any application of a rule of type $\alpha$, $\beta$ or $\delta$ performed on $\mathcal R$ at a stage $\ge k$, then $\psi \in \Pi_j$ for every $j \ge k$.
\item[\textup{(ii)}] In particular every literal, and every formula of type $\gamma$, that appears in $\Pi_k$ belongs to $\Pi_j$ for every $j \ge k$.
\end{enumerate}
\end{lem}
\begin{proof}
(i) By induction on $j - k$. Each rule removes from the parent block its principal formula, in the cases $\alpha$, $\beta$, $\delta$, or nothing at all, in the case $\gamma$; the residual context is transmitted unchanged to the child lying on $\mathcal R$. A formula that is never principal for a rule of the first three types therefore passes unchanged through every step. (ii) No rule has a literal as principal formula, so literals fall under (i); and a $\gamma$ formula, even when principal, is retained by its own rule, so it too falls under (i) as stated.
\end{proof}

The point of stating (i) with the restriction to $\alpha$, $\beta$ and $\delta$, rather than to all rules, is precisely the case of the $\gamma$ formulae, which are principal and persist nonetheless; the verification of clause (H4) in the completeness proof appeals to exactly that.

\subsection{Completeness}

We argue by contraposition, following the scheme of Smullyan's Hintikka model: a branch that refuses to close is read as a description of a model, and the description is complete enough to be realized. The notion that makes this precise is the following, in the form adapted to the relational fragment.

\begin{dfn}[First-order Hintikka set]\label{def:hintikka}
Let $C$ be a non-empty countable set of constants. A set $H$ of sentences whose closed terms belong to $C$ is a \emph{Hintikka set over $C$} if:
\begin{enumerate}
\item[\textup{(H1)}] no atomic formula occurs in $H$ together with its own negation;
\item[\textup{(H2)}] if a formula of type $\alpha$ is in $H$, both its components are in $H$;
\item[\textup{(H3)}] if a formula of type $\beta$ is in $H$, at least one of its components is in $H$;
\item[\textup{(H4)}] if $\forall x\,\varphi \in H$ then $\varphi[x/a] \in H$ for every $a \in C$, and if $\lnot\exists x\,\varphi \in H$ then $\lnot\varphi[x/a] \in H$ for every $a \in C$;
\item[\textup{(H5)}] if $\exists x\,\varphi \in H$ then $\varphi[x/a] \in H$ for some $a \in C$, and if $\lnot\forall x\,\varphi \in H$ then $\lnot\varphi[x/a] \in H$ for some $a \in C$.
\end{enumerate}
\end{dfn}

\begin{lem}[Hintikka's model lemma]\label{lem:hintikka}
Every Hintikka set $H$ over $C$ is satisfiable, in a model whose domain is $C$.
\end{lem}
\begin{proof}
Let $M_H = (D_H, i_H)$ with $D_H := C$, which is non-empty; interpret each parameter as itself, $i_H(a) := a$, so that $v(a) = a$ for $a \in C$, and put $(a_1,\ldots,a_n) \in i_H(P) :\Leftrightarrow P(a_1,\ldots,a_n) \in H$. The constants of $\Ll$ outside $C$ are interpreted arbitrarily; by Lemma~\ref{lem:coincidenza} this does not affect the sentences whose constants lie in $C$.

We show by induction on $\dg{\varphi}$, for every sentence $\varphi$ with constants in $C$, that $\varphi \in H$ implies $M_H \vDash \varphi$ and $\lnot\varphi \in H$ implies $M_H \nvDash \varphi$. The induction hypothesis applies to instances $\psi[x/a]$ of proper subformulae because $\dg{\psi[x/a]} = \dg{\psi}$.

If $\varphi = P(a_1,\ldots,a_n)$ is atomic and $\varphi \in H$, then $M_H \vDash \varphi$ by construction; if $\lnot\varphi \in H$, then $\varphi \notin H$ by (H1), so $M_H \nvDash \varphi$.

For the step we distinguish the principal sign. \emph{Negation}: if $\lnot\psi \in H$, the second half of the induction hypothesis on $\psi$ gives $M_H \nvDash \psi$, that is $M_H \vDash \lnot\psi$; if $\lnot\lnot\psi \in H$ then $\psi \in H$ by (H2), and the first half gives $M_H \vDash \psi$, that is $M_H \nvDash \lnot\psi$. \emph{Conjunction}: if $\psi\land\chi \in H$ then $\psi,\chi \in H$ by (H2) and $M_H \vDash \psi\land\chi$; if $\lnot(\psi\land\chi) \in H$ then $\lnot\psi \in H$ or $\lnot\chi \in H$ by (H3), and $M_H \nvDash \psi\land\chi$. \emph{Disjunction} and \emph{implication} are treated in the same way, the roles of (H2) and (H3) exchanging according as the formula is of type $\alpha$ or $\beta$.

\emph{Universal quantification}, $\varphi = \forall x\,\psi$. If $\forall x\,\psi \in H$, clause (H4) gives $\psi[x/a] \in H$ for every $a \in C$, whence $M_H \vDash \psi[x/a]$ by the induction hypothesis. Now every element $d$ of $D_H = C$ is a constant $a$ with $v(a) = a = d$, so Lemma~\ref{lem:sostituzione} turns each of these facts into $M_H \vDash \psi[x/c_d]$; as $d$ ranges over $D_H$ the universal clause yields $M_H \vDash \forall x\,\psi$. If $\lnot\forall x\,\psi \in H$, clause (H5) gives $\lnot\psi[x/a] \in H$ for some $a \in C$, hence $M_H \nvDash \psi[x/a]$ and, again by Lemma~\ref{lem:sostituzione}, $M_H \nvDash \psi[x/c_a]$: the element $a$ witnesses the failure of the universal clause. \emph{Existential quantification} is dual, (H5) providing the witness and (H4) the exhaustion of the domain.

Both universal steps depend on the identity $D_H = C$: every element of the domain is named by a parameter, so that the instances on parameters supplied by (H4) exhaust the witnesses that the quantifier clauses require.
\end{proof}

\begin{rem}[The case with function symbols]\label{oss:herbrand}
If $\Fun \neq \varnothing$, the correct canonical domain is the Herbrand universe $\Ter^0$ of all closed terms, with $i_H(f)(t_1,\ldots,t_n) := f(t_1,\ldots,t_n)$. The universal step then requires $\psi[x/t] \in H$ for every closed term $t$, not merely for every constant, so that clause (H4), the $\gamma$ rule and Definition~\ref{def:parametri} must be reformulated with $\adown$ ranging over $\Ter^0$. With this modification, and with the corresponding notion of fair strategy, soundness and completeness carry over unchanged, as do the results of Section~\ref{sec:strutturale}; what fails is only the finiteness of the set of terms, and with it Proposition~\ref{prop:crescita}. We keep to the relational fragment in the body so as not to burden the notation.
\end{rem}

\begin{thm}[Completeness]\label{thm:completezza}
If $\Pi$ is not satisfiable, then every tableau for $\Pi$ built according to a fair strategy closes. Consequently, if $\Gamma \segue \Delta$ is valid, every fair tableau for the associated block closes.
\end{thm}
\begin{proof}
By contraposition: if a fair tableau $T$ for $\Pi$ does not close, then $\Pi$ is satisfiable. Two cases arise.

\emph{Case 1: $T$ has a finite branch $\mathcal R = \Pi_0,\ldots,\Pi_k$ whose last block is completed open.} Put $H := \Pi_0 \cup \cdots \cup \Pi_k$ and $C := \mathrm{Par}(\mathcal R)$, non-empty by Definition~\ref{def:parametri}. Fairness plays no role here, the verification resting on the completedness of $\Pi_k$. For (H1): literals persist to $\Pi_k$ by Lemma~\ref{lem:persistenza}(ii), so an atom together with its negation in $H$ would make $\Pi_k$ closed. For (H2), (H3) and (H5): let $\psi \in \Pi_j$ be of type $\alpha$, $\beta$ or $\delta$. Since the rules of these types remove their principal formula and every application of them is productive, $\psi$ cannot belong to the completed block $\Pi_k$, for otherwise its rule would still be productively applicable; by Lemma~\ref{lem:persistenza}(i), then, $\psi$ was principal at some stage between $j$ and $k$, and that application placed into the branch both components of $\psi$ if it is $\alpha$, the component lying on $\mathcal R$ if it is $\beta$, and the instance on a fresh parameter of $C$ if it is $\delta$. For (H4): a $\gamma$ formula persists to $\Pi_k$ by Lemma~\ref{lem:persistenza}(ii), and for every $a \in C$ the $\gamma$ application on $a$ is licit at $\Pi_k$, hence unproductive by completedness, which says precisely that the instance already belongs to $\Pi_k \subseteq H$.

\emph{Case 2: no branch of $T$ terminates in a completed open block.} Since $T$ does not close, not every branch terminates in a closed block either; were $T$ finite, every leaf would be a terminated block, closed or completed open, and having excluded the second $T$ would be closed. Hence $T$ is infinite, and being finitely branching it possesses an infinite branch $\mathcal R = \Pi_0, \Pi_1, \ldots$ by K\"onig's lemma \cite{Konig1927}. No block of $\mathcal R$ is closed, a closed block being a leaf. Put $H := \bigcup_k \Pi_k$ and $C := \mathrm{Par}(\mathcal R)$. Clause (H1) holds as in Case 1: two complementary literals in $H$ would, by persistence, belong to a common $\Pi_n$, which would then be closed. Clauses (H2) and (H3) hold by (F1), every $\alpha$ or $\beta$ formula on $\mathcal R$ being decomposed at a finite stage, and (H5) likewise. Clause (H4) holds by (F2), every $\gamma$ formula being instantiated on every parameter of $\mathcal R$; here the retention of the principal formula in the $\gamma$ rules is what keeps it available for all the instantiations required.

In both cases $H$ is a Hintikka set over $C$, and every formula of $H$ is a sentence, since the root consists of sentences and every rule replaces sentences by sentences. Lemma~\ref{lem:hintikka} provides $M_H \vDash H \supseteq \Pi$, so $\Pi$ is satisfiable. The claim about sequents follows from \eqref{eq:cardine}.
\end{proof}

\begin{cor}[Adequacy]\label{cor:adeguatezza}
$\Gamma \segue \Delta$ is valid if and only if every fair tableau for $\Pi = \Gamma \cup \lnot[\Delta]$ closes. In particular $\varphi$ is a logical truth if and only if the tableau rooted at $\{\lnot\varphi\}$ closes.
\end{cor}

\subsection{Compactness and the countable model property}

Blocks are finite, so the method does not apply as it stands to an infinite set of sentences. To reach compactness we extend the construction, following Smullyan \cite{Smullyan1968}, by feeding the sentences of a countable set into the branches one at a time. The extension is used only here, and we keep its bookkeeping steps distinct from the rules of the calculus.

Both of the results that follow are proved for first-order tableaux in \cite[Ch.~V, \S4]{Smullyan1968}; we give them here in the form the block calculus takes.

\begin{cor}[Countable compactness]\label{cor:compattezza}
If every finite subset of a countable set $\Sigma$ of sentences is satisfiable, then $\Sigma$ is satisfiable.
\end{cor}
\begin{proof}
For finite $\Sigma$ the claim is trivial; let $\Sigma = \{\sigma_0, \sigma_1, \ldots\}$. Two stipulations. First, adjoin to the language a countable stock $\{d_0, d_1, \ldots\}$ of constants foreign to $\Sigma$, and require every $\delta$ application below to draw its fresh constant from this reserve; every model of a subset of $\Sigma$ extends to the enlarged language by interpreting the reserved constants arbitrarily, and by Lemma~\ref{lem:coincidenza} the extension alters no truth value of a sentence of $\Sigma$. Second, modify the procedure of Proposition~\ref{prop:esistenza-equa} at one point: the root is $\{\sigma_0\}$, and at the end of stage $n$ the sentence $\sigma_n$ is adjoined to the leaf block of every open branch. We call these insertions \emph{$\Sigma$-steps}; they introduce hypotheses rather than decompose them, and are used in this proof alone. Fairness is understood as in Definition~\ref{def:equa}, relative to all the formulae that appear.

\emph{Case 1: at no stage is every branch closed.} Since a $\Sigma$-step extends every open branch at every stage, no open branch is ever completed, and the tree is infinite; being finitely branching it has, by K\"onig's lemma, an infinite branch $\mathcal R$, none of whose blocks is closed. Let $H$ be the union of the blocks of $\mathcal R$ and $C = \mathrm{Par}(\mathcal R)$. That $H$ is a Hintikka set over $C$ is verified exactly as in Case 2 of Theorem~\ref{thm:completezza}: the $\Sigma$-steps remove no formula, so Lemma~\ref{lem:persistenza} continues to hold, and fairness discharges (H2)--(H5). Moreover $\Sigma \subseteq H$, since for each $n$ the branch is open at stage $n$ and the $\Sigma$-step of that stage adjoined $\sigma_n$ to its leaf. By Lemma~\ref{lem:hintikka}, $M_H \vDash H \supseteq \Sigma$.

\emph{Case 2: at some stage every branch is closed.} We show that this contradicts the hypothesis. A tree all of whose branches close at a finite stage is finite, by the argument of Lemma~\ref{lem:finitezza}; in it only finitely many $\Sigma$-steps have been performed, say those adjoining $\sigma_0, \ldots, \sigma_N$, and $\Sigma_0 := \{\sigma_0,\ldots,\sigma_N\}$ is a finite subset of $\Sigma$. By hypothesis there is $M \vDash \Sigma_0$, regarded as extended to the reserved constants. Descend the closed tree along satisfiable blocks as in Theorem~\ref{thm:correttezza}, updating the model as one goes: at a rule step Lemma~\ref{lem:conservazione} provides a satisfiable child and, in the $\delta$ cases, replaces the current model by one differing only in the interpretation of a reserved constant fresh for the branch; at a $\Sigma$-step from $\Pi'$ to $\Pi' \cup \{\sigma_n\}$ with $n \le N$, the current model $M'$ satisfies $\Pi'$ by construction and satisfies $\sigma_n$ as well, since $M'$ differs from $M$ only on reserved constants, none of which occurs in $\sigma_n$, so that Lemma~\ref{lem:coincidenza} transfers $M \vDash \sigma_n$ to $M'$. The descent ends in a leaf both closed and satisfiable, against Proposition~\ref{prop:chiusura}.
\end{proof}

\begin{cor}[Countable model property; downward L\"owenheim--Skolem]\label{cor:LS}
Every satisfiable block has a model whose domain is countable; and every countable set of sentences that is satisfiable has a model whose domain is countable.
\end{cor}
\begin{proof}
Let $\Pi$ be satisfiable. By Theorem~\ref{thm:correttezza} no tableau for $\Pi$ closes, so a fair tableau for $\Pi$ -- which exists by Proposition~\ref{prop:esistenza-equa} -- falls under one of the two cases of Theorem~\ref{thm:completezza}, and the model produced by Lemma~\ref{lem:hintikka} has domain $C = \mathrm{Par}(\mathcal R)$, a set of constants, hence countable. For a countable $\Sigma$, Case 1 of Corollary~\ref{cor:compattezza} produces a model of $\Sigma$ with the same property, and Case 2 is impossible when $\Sigma$ is satisfiable.
\end{proof}

\begin{rem}[Where the non-constructive step lies]\label{oss:wkl}
Both Theorem~\ref{thm:completezza} and Corollary~\ref{cor:compattezza} use K\"onig's lemma for finitely branching trees, and nothing stronger. This is not an artefact of the presentation: completeness for countable languages is provable in $\mathrm{WKL}_0$ and is, over $\mathrm{RCA}_0$, equivalent to it \cite{Simpson2009}. The contrast with Section~\ref{sec:strutturale} is worth recording, since the admissibility of cut proved there is entirely finitary.
\end{rem}

\noindent The $\Sigma$-steps, being insertions of hypotheses, fall outside the analyticity analysis of the next section, which concerns tableaux built by the rules of the calculus alone.

\section{Analyticity and termination}\label{sec:analiticita}

A refutation calculus is \emph{analytic}, in the tradition that reaches from Gentzen to contemporary proof theory, if every proof proceeds by decomposing and never by composing: nothing enters the proof that is not already contained in its conclusion. In the sequent calculus the requirement is formalized in the subformula property, and it admits the characterization that a calculus enjoys that property when cut is admissible and, in every other rule, the formulae of the premisses are subformulae of the formulae of the conclusion. The first half is Theorem~\ref{thm:taglio}. The second holds here in the form appropriate to first-order logic, where the literal notion of subformula must give way to a notion closed under instantiation.

\subsection{Generalized subformulae}

Two notions of subformula are needed, and it is the gap between them that measures the price of the passage to the first order.

\begin{dfn}[Subformula]\label{def:sottoformula}
$\Sub$ is the smallest reflexive and transitive relation on $\Frml$ with $\varphi \in \Sub(\lnot\varphi)$, $\varphi,\psi \in \Sub(\varphi \star \psi)$ for $\star \in \{\land,\lor,\rightarrow\}$, and $\varphi \in \Sub(Qx\,\varphi)$ for $Q \in \{\forall,\exists\}$. We put $\Sub(\Pi) := \bigcup_{\varphi\in\Pi}\Sub(\varphi)$.
\end{dfn}

The rule $\forall$ produces from $\forall x\,\varphi$ the instance $\varphi[x/\adown]$, which is not a subformula of $\forall x\,\varphi$ in this sense, containing as it does a constant where the original has a variable. The notion has to be widened, as Gentzen already saw.

\begin{dfn}[Generalized subformula]\label{def:sottoformula-gen}
$\Subg(\Pi)$ is the smallest set of formulae such that $\Sub(\Pi) \subseteq \Subg(\Pi)$; if $\forall x\,\varphi \in \Subg(\Pi)$ or $\exists x\,\varphi \in \Subg(\Pi)$ then $\varphi[x/t] \in \Subg(\Pi)$ for every closed term $t$ with $\Leg\,\varphi[x/t]$; and $\Subg(\Pi)$ is closed under immediate propositional components. Every $\psi \in \Subg(\Pi)$ is an instance of a subformula of some $\varphi \in \Pi$.
\end{dfn}

\begin{thm}[Generalized subformula property]\label{thm:subformula}
Every block occurring in a tableau for $\Pi$ is contained in $\Subg(\Pi)$.
\end{thm}
\begin{proof}
By induction on the depth of the node. The root satisfies $\Pi \subseteq \Subg(\Pi)$. Let $\Pi' \subseteq \Subg(\Pi)$ and let a rule with principal formula $\varphi_0 \in \Pi'$ be applied. For the rules $\alpha$ and $\beta$ the formulae introduced are immediate subformulae of $\varphi_0$, or negations of such, both covered by the closure conditions, and the context $\Pi' \setminus \{\varphi_0\}$ already lies in $\Subg(\Pi)$. For a $\gamma$ rule the principal formula is retained and the instance introduced is $\varphi[x/\adown]$, respectively $\lnot\varphi[x/\adown]$, which lies in $\Subg(\Pi)$ by closure under instantiation and, in the negated form, under components. For a $\delta$ rule the same clauses apply to $\varphi[x/\aup]$.
\end{proof}

\begin{cor}[Analyticity]\label{cor:analiticita}
The block calculus is analytic: cut is admissible (Theorem~\ref{thm:taglio}) and every refutation uses only generalized subformulae of the root (Theorem~\ref{thm:subformula}). No block of a tableau for $\Pi$ contains a formula extraneous to $\Subg(\Pi)$: in particular no cut formula and no auxiliary lemma occurs in it. In the propositional fragment, where $\gamma$ and $\delta$ are absent, $\Subg(\Pi)$ reduces to $\Sub(\Pi)$ closed under negation of components, and the subformula property holds in the literal sense.
\end{cor}

\subsection{Controlled loss and the growth of parameters}

In first-order logic the literal subformula property cannot hold for any sound and complete calculus, the quantifier rules being obliged to introduce instances $\varphi[x/t]$ that contain terms absent from the principal formula. The loss is nonetheless controlled, in the following precise sense.

\begin{prop}[Localization of the loss]\label{prop:localizzazione}
The only formulae appearing in a tableau for $\Pi$ that are not literal subformulae, or negations of literal subformulae, of formulae of $\Pi$ are instances $\varphi[x/t]$ and $\lnot\varphi[x/t]$ of quantified subformulae of formulae of $\Pi$, with $t$ a closed term occurring on the branch.
\end{prop}
\begin{proof}
Immediate from Theorem~\ref{thm:subformula} and Definition~\ref{def:sottoformula-gen}: the elements of $\Subg(\Pi)\setminus\Sub(\Pi)$ are exactly those produced by the instantiation clause, possibly negated.
\end{proof}

\begin{prop}[Growth of parameters]\label{prop:crescita}
In the relational fragment, let $\mathcal R$ be a branch of a tableau for $\Pi$ and let $\Cost(\Pi)$ be the set of constants occurring in $\Pi$. Then
\[
|\mathrm{Par}(\mathcal R)| \;\le\; |\Cost(\Pi)| \;+\; \#\{\delta \text{ applications on } \mathcal R\} \;+\; 1 ,
\]
the final unit accounting for the reserved parameter $a_0$. Every $\delta$ application introduces at most one new parameter, and the rules $\alpha, \beta, \gamma$ introduce none. In particular, if only finitely many $\delta$ applications are performed on $\mathcal R$, the set of terms occurring on $\mathcal R$ is finite.
\end{prop}
\begin{proof}
The rules $\alpha$ and $\beta$ act on truth-functional connectives and introduce no constant; the $\gamma$ rule instantiates on a parameter already available. The only rule that introduces a constant is $\delta$, and it introduces one per application -- or none, in the degenerate case in which the quantified variable does not occur free in the matrix, discussed after Definition~\ref{def:nuova}. The count follows.
\end{proof}

\begin{rem}[The functional case]\label{oss:funzionale}
With function symbols (Remark~\ref{oss:herbrand}) the linear control fails: a single constant $\aup$ generates through the functors the Herbrand universe $\{\aup, f(\aup), f(f(\aup)), \ldots\}$. Theorem~\ref{thm:subformula} survives, all these instances being generalized subformulae of the root, so analyticity is preserved and it is only the finiteness of the set of terms that is lost.
\end{rem}

\subsection{A terminating fragment}

Proposition~\ref{prop:crescita} bounds the parameters in terms of the $\delta$ applications, and says nothing about how many of those there are. On a syntactically delimited class of blocks the number is bounded outright, and the construction terminates. We first isolate one point of hygiene. Call a construction \emph{non-repeating} if on each branch no pair $\langle \psi, a \rangle$, with $\psi$ of type $\gamma$ and $a$ a parameter, is discharged twice. The procedure of Proposition~\ref{prop:esistenza-equa} is non-repeating, each such pair being appended to a queue exactly once. The proviso is not idle. A $\gamma$ instance may be introduced, decomposed by a later rule and thereby removed from the block, at which point the $\gamma$ application that produced it becomes productive again; a construction free to seize that opportunity would loop, and would do so on blocks as simple as $\{\forall x\,(P(x) \land Q(x))\}$.

\begin{dfn}[$\delta$-stratified block]\label{def:stratificato}
A block $\Pi$ is \emph{$\delta$-stratified} if no formula of type $\gamma$ belonging to $\Subg(\Pi)$ has a formula of type $\delta$ among its subformulae.
\end{dfn}

The condition is met whenever every $\gamma$ formula of $\Subg(\Pi)$ has a quantifier-free matrix, and in particular in the propositional fragment and in the block translation of the Bernays--Sch\"onfinkel prefix class.

\begin{thm}[Termination and decidability]\label{thm:terminazione}
Let $\Pi$ be $\delta$-stratified. Then every non-repeating fair construction for $\Pi$ is a finite tableau. Consequently the validity of a sequent whose associated block is $\delta$-stratified is decidable, and the procedure of Proposition~\ref{prop:esistenza-equa} decides it.
\end{thm}
\begin{proof}
Fix a branch $\mathcal R$. We show first that only finitely many $\delta$ applications occur on it. Call a formula occurrence on $\mathcal R$ a \emph{$\gamma$-descendant} if the chain of rule applications that produced it passes through a $\gamma$ instance. By $\delta$-stratification no $\gamma$-descendant is of type $\delta$: an instance of a $\gamma$ formula $\psi$ has among its subformulae only instances of subformulae of $\psi$, and $\psi$ has no $\delta$ subformula. Hence every $\delta$ formula occurring on $\mathcal R$ is produced from a formula of $\Pi$ by a chain of applications of type $\alpha$, $\beta$ or $\delta$ alone. Each such application consumes one formula occurrence and produces at most two, of strictly smaller weight by Lemma~\ref{lem:peso}; along a single branch the resulting forest of descendants of $\Pi$ has therefore depth at most $\max_{\varphi\in\Pi}\wt{\varphi}$ and at most $\sum_{\varphi \in \Pi} 2^{\wt{\varphi}}$ nodes, which bounds the number of $\delta$ applications on $\mathcal R$.

By Proposition~\ref{prop:crescita} the set $K := \mathrm{Par}(\mathcal R)$ is then finite. Every formula occurring on $\mathcal R$ lies in $\Subg(\Pi)$ by Theorem~\ref{thm:subformula}, and all its constants lie in $K$; since $\Sub(\Pi)$ is finite and each of its elements has finitely many free variables, the set $F$ of formulae that can occur on $\mathcal R$ is finite.

We now count the applications performed on $\mathcal R$. Those of type $\gamma$ are at most $|F| \cdot |K|$, since the construction is non-repeating and each is determined by a pair $\langle \psi, a\rangle$ with $\psi \in F$ and $a \in K$. Those of the remaining types consume a formula occurrence and produce at most two, each of strictly smaller degree. Consider therefore the forest whose roots are the formulae of $\Pi$ together with the instances introduced by the $\gamma$ applications -- finitely many, by what precedes -- and in which the children of an occurrence are the formulae that an application of type $\alpha$, $\beta$ or $\delta$ produces from it. Every node of this forest has at most two children, and the weight strictly decreases from a node to its children by Lemma~\ref{lem:peso}, so each tree of the forest has depth at most $\max_{\varphi \in F}\wt{\varphi}$ and is finite. The applications of type $\alpha$, $\beta$ and $\delta$ on $\mathcal R$ are in bijection with the internal nodes of the forest, hence finite in number as well. The branch $\mathcal R$ is therefore finite. Since the tree is finitely branching and all its branches are finite, it is finite by K\"onig's lemma.

For decidability: given $\Pi$ $\delta$-stratified, run the procedure of Proposition~\ref{prop:esistenza-equa}, which is fair and non-repeating. It halts, by what precedes, on a finite tableau, and that tableau is closed if and only if $\Pi$ is unsatisfiable, by Theorems~\ref{thm:correttezza} and \ref{thm:completezza}.
\end{proof}

\begin{cor}[Propositional bound]\label{cor:prop-bound}
Put $W := \sum_{\varphi\in\Pi}\wt{\varphi}$. In the propositional fragment every tableau for $\Pi$ has height at most $W$ and at most $2^{W}$ leaves; in particular $W \le 2\sum_{\varphi\in\Pi}\dg{\varphi}$.
\end{cor}
\begin{proof}
By Lemma~\ref{lem:peso} the total weight of a block decreases by at least one at every propositional step, and it is non-negative; this bounds the height. Each step produces at most two children, which bounds the number of leaves. The final inequality holds because $\wt{\varphi} \le 2\dg{\varphi}$ for every $\varphi$.
\end{proof}

\subsection{Comparison with Smullyan's tableaux}

In Smullyan's presentation each node carries a single formula, the hierarchy $\alpha,\beta,\gamma,\delta$ is the same, and a $\gamma$ formula remains available for successive instantiations by returning to the node that carries it. The two presentations share cut-freeness and the generalized subformula property, the same controlled loss at the quantifier level, and the same closure criterion. They differ in the organization of the information: the block calculus transports the entire block rather than a single formula, and it realizes the repeatability of the $\gamma$ rules by the persistence of the principal formula rather than by returning to the nodes. No loss of analyticity is introduced by the block form: the two hierarchies of generalized subformulae coincide, and the replacement of the literal property by the generalized one is common to both and, as we have seen, unavoidable.

\section{The sequent calculus $\mathrm{LK}$ and the correspondence}\label{sec:LK}

The block calculus is not a merely notational hybrid: it stands in structural correspondence with Gentzen's $\mathrm{LK}$. We present the rules of $\mathrm{LK}$ for $\Ll$, in a formulation modelled on Galvan \cite{Galvan2015}, and then prove the correspondence in both directions. Since cut is admissible on the side of blocks, the direction from $\mathrm{LK}$ to tableaux no longer needs to assume its premiss cut-free, and Gentzen's Hauptsatz follows by transfer.

\subsection{The system}

Sequents are pairs of finite sets of sentences; $\alpha$ is a theorem of classical first-order logic exactly when $\segue \alpha$ is derivable. Contraction is built into the notation, so the explicit structural rules are weakening and cut.

\medskip
\noindent \textbf{Axiom.} \quad $\alpha \segue \alpha \; (\mathrm{Ax})$

\medskip
\noindent \textbf{Structural rules.}
\[
\frac{\Gamma \segue \Delta}{\alpha, \Gamma \segue \Delta} \; \text{Ws}
\qquad
\frac{\Gamma \segue \Delta}{\Gamma \segue \Delta, \alpha} \; \text{Wd}
\qquad
\frac{\Gamma \segue \Delta, \mu \quad \mu, \Theta \segue \Sigma}{\Gamma, \Theta \segue \Delta, \Sigma} \; \text{Cut}
\]

\medskip
\noindent \textbf{Propositional operational rules.}
\[
\frac{\Gamma \segue \Delta, \alpha}{\lnot \alpha, \Gamma \segue \Delta} \; \text{$\lnot$s}
\qquad
\frac{\alpha, \Gamma \segue \Delta}{\Gamma \segue \Delta, \lnot \alpha} \; \text{$\lnot$d}
\]
\[
\frac{\alpha, \beta, \Gamma \segue \Delta}{\alpha \land \beta, \Gamma \segue \Delta} \; \text{$\land$s}
\qquad
\frac{\Gamma \segue \Delta, \alpha \quad \Gamma \segue \Delta, \beta}{\Gamma \segue \Delta, \alpha \land \beta} \; \text{$\land$d}
\]
\[
\frac{\alpha, \Gamma \segue \Delta \quad \beta, \Gamma \segue \Delta}{\alpha \lor \beta, \Gamma \segue \Delta} \; \text{$\lor$s}
\qquad
\frac{\Gamma \segue \Delta, \alpha, \beta}{\Gamma \segue \Delta, \alpha \lor \beta} \; \text{$\lor$d}
\]
\[
\frac{\Gamma \segue \Delta, \alpha \quad \beta, \Gamma \segue \Delta}{\alpha \rightarrow \beta, \Gamma \segue \Delta} \; \text{$\rightarrow$s}
\qquad
\frac{\alpha, \Gamma \segue \Delta, \beta}{\Gamma \segue \Delta, \alpha \rightarrow \beta} \; \text{$\rightarrow$d}
\]

\medskip
\noindent \textbf{Quantifier rules.} We work with closed sentences in the relational fragment, so instantiation is on constants: $\alpha(a) \equiv \alpha(x)[x/a]$.
\[
\frac{\alpha(a), \Gamma \segue \Delta}{\forall x\,\alpha(x), \Gamma \segue \Delta} \; \text{$\forall$s}
\qquad
\frac{\Gamma \segue \Delta, \alpha(a)}{\Gamma \segue \Delta, \forall x\,\alpha(x)} \; \text{$\forall$d}^{\ast}
\]
\[
\frac{\alpha(a), \Gamma \segue \Delta}{\exists x\,\alpha(x), \Gamma \segue \Delta} \; \text{$\exists$s}^{\ast}
\qquad
\frac{\Gamma \segue \Delta, \alpha(a)}{\Gamma \segue \Delta, \exists x\,\alpha(x)} \; \text{$\exists$d}
\]

The rules marked with an asterisk are \emph{critical}: their \emph{proper parameter} $a$ must not occur in the conclusion. In Gentzen's original formulation the same role is played by a free variable, the Eigenvariable; the parameter-based formulation adopted here goes back to Smullyan \cite{Smullyan1968} and matches the closed-sentence setting of blocks. The critical condition is, on the sequent side, the exact counterpart of the freshness condition on $\aup$: both express the introduction of a witness that is fixed but generic. In $\forall$s and $\exists$d the instantiating constant is arbitrary, and these correspond to the $\gamma$ rules. Cut-free derivability is written $\vdash_{\mathrm{cf}}$, derivability with Cut $\vdash_{\mathrm{LK}}$.

\subsection{Three lemmas on $\mathrm{LK}$}

The proof of the correspondence manipulates $\mathrm{LK}$ derivations, and it needs the standard liberty of renaming proper parameters, together with the invertibility of one rule. These facts are folklore; since it is exactly at these junctures that informal treatments of the correspondence pass over the details in silence, we prove them.

\begin{lem}[Renaming of constants]\label{lem:rinomina}
Let $\mathcal D$ be a cut-free derivation of $\Gamma \segue \Delta$ and let $a, c$ be constants neither of which is the proper parameter of a critical application in $\mathcal D$. Then $\mathcal D[a/c]$ is a cut-free derivation of $(\Gamma \segue \Delta)[a/c]$, of the same height.
\end{lem}
\begin{proof}
By induction on the height of $\mathcal D$. An axiom $\theta \segue \theta$ becomes $\theta[a/c] \segue \theta[a/c]$, again an axiom. For a non-critical rule the uniform replacement of a constant by a constant commutes with the formation of instances -- in particular $\alpha(x)[x/b][a/c] = \alpha(x)[a/c][x/b[a/c]]$, so an application of $\forall$s or $\exists$d with instantiating constant $b$ becomes one with instantiating constant $b[a/c]$ -- and no side condition is at stake. For a critical application with proper parameter $p$: by hypothesis $p \notin \{a,c\}$, so $p$ is untouched and remains the proper parameter; its side condition survives, since $p$ was absent from the old conclusion and the replacement, trading occurrences of $a$ for occurrences of $c \neq p$, cannot introduce $p$ into it.
\end{proof}

\begin{lem}[Pure-parameter form]\label{lem:puri}
Every cut-free derivation of $\Gamma \segue \Delta$ can be transformed into a cut-free derivation of the same endsequent and the same height in which distinct critical applications have distinct proper parameters, each proper parameter occurs only in the subderivation ending with the premiss of its application, and all proper parameters avoid a finite set $X$ of constants fixed in advance.
\end{lem}
\begin{proof}
By induction on the height of $\mathcal D$. If the last rule is not critical, transform the immediate subderivations by the induction hypothesis, enlarging $X$ progressively so that the parameters chosen in one subderivation avoid those chosen in the other. If the last rule is critical -- say $\forall$d with premiss $\Gamma_0 \segue \Delta_0, \alpha(p)$ and conclusion $\Gamma_0 \segue \Delta_0, \forall x\,\alpha(x)$, with $p$ absent from the conclusion -- transform the premiss subderivation $\mathcal D'$ by the induction hypothesis, with proper parameters avoiding $X \cup \{p\}$, and choose $q$ occurring nowhere in $\mathcal D'$ and outside $X$, which exists because $\Cost$ is countable while $\mathcal D'$ and $X$ are finite. In the transformed $\mathcal D'$ neither $p$ nor $q$ is a proper parameter, so Lemma~\ref{lem:rinomina} applies and $\mathcal D'[p/q]$ derives $\Gamma_0 \segue \Delta_0, \alpha(q)$, the contexts and $\alpha(x)$ being unaffected because $p$ occurs in none of them. An application of $\forall$d with proper parameter $q$ returns the original conclusion. The case $\exists$s is symmetric, and renaming alters no height.
\end{proof}

\begin{lem}[Invertibility of $\lnot$s]\label{lem:inversione}
If $\vdash_{\mathrm{cf}} \lnot\alpha, \Gamma \segue \Delta$ then $\vdash_{\mathrm{cf}} \Gamma \segue \Delta, \alpha$.
\end{lem}
\begin{proof}
If $\lnot\alpha \in \Gamma$ the required sequent follows from the hypothesis by Wd, so we may assume $\lnot\alpha \notin \Gamma$, the comma denoting set-theoretic union. We induct on the height of a cut-free derivation $\mathcal D$, by cases on its last rule.

If $\mathcal D$ is an axiom $\theta \segue \theta$, the antecedent is $\{\theta\}$ and contains $\lnot\alpha$, so $\theta = \lnot\alpha$, $\Gamma = \varnothing$, $\Delta = \{\lnot\alpha\}$; the sequent $\segue \lnot\alpha, \alpha$ is derived from $\alpha \segue \alpha$ by $\lnot$d. If the last rule is Ws introducing $\lnot\alpha$, its premiss is $\Gamma \segue \Delta$ and Wd yields the claim. If it is $\lnot$s with principal formula $\lnot\alpha$, its premiss is already $\Gamma \segue \Delta, \alpha$.

In every other case the principal formula of the last rule $R$ differs from $\lnot\alpha$, and $\lnot\alpha$ occurs in the antecedent of every premiss of $R$; apply the induction hypothesis to each premiss and re-apply $R$. Two points need attention. Antecedents being sets, $\lnot\alpha$ may coincide with an active formula that $R$ consumes -- with a component $\beta_1$ of a principal $\beta_1\land\beta_2$ of $\land$s, say; in that case one application of Ws restores the missing occurrence before $R$ is re-applied, and the final antecedent is the required one, duplicates collapsing. And if $R$ is critical with proper parameter $p$, the side condition demands that $p$ not occur in $\Gamma \setminus \{\lnot\alpha\} \segue \Delta, \alpha$: it does not occur in $\Gamma$ or $\Delta$ by the original side condition, and it does not occur in $\alpha$ because $\lnot\alpha$ belonged to the original conclusion, from which $p$ was absent.
\end{proof}

\subsection{The correspondence}

A block $\Pi$ is read as the sequent $\Pi \segue$, with empty succedent; a closed block $\{p, \lnot p\}$ then becomes $p, \lnot p \segue$, which is derivable from the basic sequent $p \segue p$ by $\lnot$s and weakening. Under this reading the two calculi simulate each other, and the simulation is what the following theorem asserts.

\begin{thm}[Correspondence]\label{thm:corrispondenza}
Let $\Pi = \Gamma \cup \lnot[\Delta]$ be the block associated with $\Gamma \segue \Delta$.
\begin{enumerate}
\item[\textup{(a)}] If there is a closed tableau for $\Pi$, then $\vdash_{\mathrm{cf}} \Gamma \segue \Delta$.
\item[\textup{(b)}] If $\vdash_{\mathrm{LK}} \Gamma \segue \Delta$ -- Cut allowed -- then there is a closed tableau for $\Pi$.
\end{enumerate}
\end{thm}
\begin{proof}
(a) We prove that a block $\Pi'$ with a closed tableau satisfies $\vdash_{\mathrm{cf}} \Pi' \segue$. Statement (a) follows: from $\vdash_{\mathrm{cf}} \Gamma \cup \lnot[\Delta] \segue$, successive applications of Lemma~\ref{lem:inversione} to $\lnot\psi_1,\ldots,\lnot\psi_m$ move the $\psi_j$ to the succedent and yield $\vdash_{\mathrm{cf}} \Gamma \setminus \lnot[\Delta] \segue \Delta$; if some $\lnot\psi_j$ belonged to $\Gamma$ as well -- the situation of Remark~\ref{oss:non-iniettiva} -- finitely many applications of Ws restore it, and $\vdash_{\mathrm{cf}} \Gamma \segue \Delta$.

The claim is proved by induction on the height of the closed tableau $T$ for $\Pi'$. If $T$ has height $0$, then $\Pi'$ contains a complementary pair of literals $p, \lnot p$: from $p \segue p$ one application of $\lnot$s gives $\lnot p, p \segue$, and finitely many applications of Ws reconstruct the context $\Pi' \segue$. Otherwise let the rule applied at the root have principal formula $\varphi_0$, and put $\Pi_0 = \Pi' \setminus \{\varphi_0\}$; the subtableaux rooted at the children are closed and of smaller height, so the induction hypothesis applies to them. The eleven cases are:
\begin{enumerate}
\item[$(\land)$] IH: $\vdash_{\mathrm{cf}} \varphi, \psi, \Pi_0 \segue$. By $\land$s: $\varphi\land\psi, \Pi_0 \segue$.
\item[$(\lnot\lnot)$] IH: $\vdash_{\mathrm{cf}} \varphi, \Pi_0 \segue$. By $\lnot$d: $\Pi_0 \segue \lnot\varphi$; by $\lnot$s: $\lnot\lnot\varphi, \Pi_0 \segue$.
\item[$(\lnot\lor)$] IH: $\vdash_{\mathrm{cf}} \lnot\varphi, \lnot\psi, \Pi_0 \segue$. By Lemma~\ref{lem:inversione} twice: $\Pi_0 \segue \varphi, \psi$; by $\lor$d and $\lnot$s: $\lnot(\varphi\lor\psi), \Pi_0 \segue$.
\item[$(\lnot\rightarrow)$] IH: $\vdash_{\mathrm{cf}} \varphi, \lnot\psi, \Pi_0 \segue$. By Lemma~\ref{lem:inversione}: $\varphi, \Pi_0 \segue \psi$; by $\rightarrow$d and $\lnot$s: $\lnot(\varphi\rightarrow\psi), \Pi_0 \segue$.
\item[$(\lor)$] IH: $\vdash_{\mathrm{cf}} \varphi, \Pi_0 \segue$ and $\vdash_{\mathrm{cf}} \psi, \Pi_0 \segue$. By $\lor$s.
\item[$(\lnot\land)$] IH: $\vdash_{\mathrm{cf}} \lnot\varphi, \Pi_0 \segue$ and $\vdash_{\mathrm{cf}} \lnot\psi, \Pi_0 \segue$. By Lemma~\ref{lem:inversione}, $\land$d and $\lnot$s.
\item[$(\rightarrow)$] IH: $\vdash_{\mathrm{cf}} \lnot\varphi, \Pi_0 \segue$ and $\vdash_{\mathrm{cf}} \psi, \Pi_0 \segue$. By Lemma~\ref{lem:inversione} on the first and $\rightarrow$s with the second.
\item[$(\forall)$] IH: $\vdash_{\mathrm{cf}} \varphi[x/\adown], \forall x\,\varphi, \Pi_0 \segue$. By $\forall$s with context $\{\forall x\,\varphi\}\cup\Pi_0$ and instantiating constant $\adown$, the conclusion being $\forall x\,\varphi, \forall x\,\varphi, \Pi_0 \segue$, that is $\forall x\,\varphi, \Pi_0 \segue$. The implicit contraction of the set-theoretic reading is the counterpart of the persistence of the principal formula in the $\gamma$ rule.
\item[$(\lnot\exists)$] IH: $\vdash_{\mathrm{cf}} \lnot\exists x\,\varphi, \lnot\varphi[x/\adown], \Pi_0 \segue$. By Lemma~\ref{lem:inversione}, $\exists$d and $\lnot$s, with the same collapse of the duplicate.
\item[$(\exists)$] IH: $\vdash_{\mathrm{cf}} \varphi[x/\aup], \Pi_0 \segue$. By $\exists$s with proper parameter $\aup$: its side condition demands that $\aup$ be absent from $\exists x\,\varphi, \Pi_0 \segue$, and this is exactly the freshness of $\aup$ for the branch, which contains the root $\Pi' = \Pi_0 \cup \{\exists x\,\varphi\}$.
\item[$(\lnot\forall)$] IH: $\vdash_{\mathrm{cf}} \lnot\varphi[x/\aup], \Pi_0 \segue$. By Lemma~\ref{lem:inversione}, then $\forall$d with proper parameter $\aup$, licit by freshness, then $\lnot$s.
\end{enumerate}
No step employs Cut, so the derivations produced are cut-free.

\smallskip
(b) By induction on the height of a derivation $\mathcal D$ of $\Gamma \segue \Delta$ in $\mathrm{LK}$, we construct a closed tableau for $\Pi$. If $\mathcal D$ is the axiom $\alpha \segue \alpha$, then $\Pi = \{\alpha,\lnot\alpha\}$, which is refutable by Lemma~\ref{lem:chiusura-gen}. Otherwise we distinguish according to the last rule.
\begin{enumerate}
\item[(Ws)] From $\Gamma \segue \Delta$ to $\alpha, \Gamma \segue \Delta$: the induction hypothesis gives a closed tableau for $\Gamma \cup \lnot[\Delta]$, and Lemma~\ref{lem:wk} with $\Sigma = \{\alpha\}$ one for the required block. (Wd) is symmetric, with $\Sigma = \{\lnot\alpha\}$.
\item[($\lnot$s)] From $\Gamma \segue \Delta, \alpha$ to $\lnot\alpha, \Gamma \segue \Delta$: the block of the conclusion is identical to the block of the premiss, and the tableau of the induction hypothesis already serves.
\item[($\lnot$d)] From $\alpha, \Gamma \segue \Delta$ to $\Gamma \segue \Delta, \lnot\alpha$: apply $(\lnot\lnot)$ at the root of $\Gamma \cup \lnot[\Delta] \cup \{\lnot\lnot\alpha\}$, obtaining the block of the premiss.
\item[($\land$s), ($\lor$d), ($\rightarrow$d)] One-premiss rules, matched respectively by $(\land)$, $(\lnot\lor)$, $(\lnot\rightarrow)$: in each case the unique child is the block of the premiss.
\item[($\land$d), ($\lor$s), ($\rightarrow$s)] Two-premiss rules, matched respectively by $(\lnot\land)$, $(\lor)$, $(\rightarrow)$: the two children are the blocks of the two premisses.
\item[($\forall$s)] From $\alpha(a), \Gamma \segue \Delta$ to $\forall x\,\alpha(x), \Gamma \segue \Delta$; the required block is $B := \{\forall x\,\alpha\} \cup \Gamma \cup \lnot[\Delta]$. If $x$ is not free in $\alpha$ then $\alpha(a) = \alpha = \alpha(b)$ for every $b$, and the argument goes through with any admissible parameter. Choose $b := a$ if $a$ occurs in $B$; otherwise let $b$ be any constant occurring in $B$, or $a_0$ if $B$ contains none. In the second case $a$ occurs in none of $\Gamma, \Delta, \alpha(x)$; Lemma~\ref{lem:puri} puts the premiss subderivation in pure form with proper parameters avoiding $\{a,b\}$, and Lemma~\ref{lem:rinomina} turns it into a derivation of $\alpha(b), \Gamma \segue \Delta$ of the same height. The induction hypothesis applied to that derivation yields a closed tableau for $\{\alpha(b)\}\cup\Gamma\cup\lnot[\Delta]$, and Lemma~\ref{lem:wk} extends it to one for $\{\forall x\,\alpha, \alpha(b)\}\cup\Gamma\cup\lnot[\Delta]$. Applying at the root $B$ the $\gamma$ rule $(\forall)$ with $\adown = b$ -- licit, $b$ being a parameter of the root branch by construction -- produces exactly that block.
\item[($\exists$d)] Mirror image of the previous case, with $(\lnot\exists)$ in place of $(\forall)$.
\item[($\exists$s)] Critical, from $\alpha(p), \Gamma \segue \Delta$ to $\exists x\,\alpha(x), \Gamma \segue \Delta$ with $p$ absent from the conclusion. Apply at the root the $\delta$ rule $(\exists)$ with $\aup := p$: freshness demands that $p$ occur nowhere in the block, which is exactly the critical condition, and the unique child is the block of the premiss. If $p = a_0$, Lemmas~\ref{lem:puri} and \ref{lem:rinomina} first replace it by a constant distinct from $a_0$ and absent from the conclusion.
\item[($\forall$d)] Critical, symmetric, with the $\delta$ rule $(\lnot\forall)$ and the same proviso on $a_0$.
\item[(Cut)] From $\Gamma \segue \Delta, \mu$ and $\mu, \Theta \segue \Sigma$ to $\Gamma, \Theta \segue \Delta, \Sigma$. The blocks of the premisses are $\Gamma \cup \lnot[\Delta] \cup \{\lnot\mu\}$ and $\Theta \cup \lnot[\Sigma] \cup \{\mu\}$, and the block of the conclusion is $\Pi = \Gamma \cup \Theta \cup \lnot[\Delta] \cup \lnot[\Sigma]$. The induction hypothesis provides closed tableaux for the two former; Lemma~\ref{lem:wk} enlarges them to closed tableaux for $\Pi \cup \{\lnot\mu\}$ and $\Pi \cup \{\mu\}$, and Theorem~\ref{thm:taglio} yields a closed tableau for $\Pi$.
\end{enumerate}
In every case the tableau constructed is closed.
\end{proof}

\begin{cor}[Hauptsatz for $\mathrm{LK}$, by transfer]\label{cor:hauptsatz}
If $\vdash_{\mathrm{LK}} \Gamma \segue \Delta$ then $\vdash_{\mathrm{cf}} \Gamma \segue \Delta$.
\end{cor}

That the Hauptsatz for tableaux implies the Hauptsatz for Gentzen systems, by constructive translation in both directions, is Smullyan's observation \cite[p.~111]{Smullyan1968}; what the proof below adds is the translation for $\mathrm{LK}$ in the presentation of Section~\ref{sec:LK}, with the axiom $\alpha \segue \alpha$ and explicit weakening.
\begin{proof}
By Theorem~\ref{thm:corrispondenza}(b) the block $\Pi = \Gamma \cup \lnot[\Delta]$ has a closed tableau, and by Theorem~\ref{thm:corrispondenza}(a) that tableau translates into a cut-free derivation of $\Gamma \segue \Delta$.
\end{proof}

The route is worth describing, since it is the one along which the correspondence stops being an equivalence statement and becomes an instrument. A derivation with cut is translated into a tableau, cut being absorbed by Theorem~\ref{thm:taglio}; the tableau, having no cut to begin with, is translated back into a derivation, which is therefore cut-free. Every step is finitary: Theorem~\ref{thm:taglio} is proved syntactically, and no appeal is made to soundness, to completeness, or to K\"onig's lemma. Two other routes to the same conclusion are worth setting beside it. One is semantic and takes three lines -- cut-free $\mathrm{LK}$ is complete and $\mathrm{LK}$ with cut is sound for the classical semantics, so a derivable sequent is valid and therefore derivable without cut -- but it produces no derivation, only the knowledge that one exists \cite[Ch.~1]{Franks2026}. The other is Gentzen's own, an algorithm internal to $\mathrm{LK}$ that rewrites a derivation until every cut has disappeared \cite{Gentzen1935, Buss1998}. The route through blocks is constructive like the second and short like the first; what it costs is the detour through a second calculus. The same pattern -- eliminate the cut where it is cheapest to eliminate, then transfer -- governs the propositional treatment of the companion paper \cite{cuconato-proofteoria}, where cut elimination for a $\mathrm{G0}$-style calculus is obtained by passing through its $\mathrm{G3}$-style counterpart.

\begin{rem}[Correspondence, not isomorphism]\label{oss:non-iso}
Theorem~\ref{thm:corrispondenza} establishes a simulation preserving derivability in both directions; it does not assert a structure-preserving bijection between closed tableaux and $\mathrm{LK}$ derivations. The map is not injective, for three reasons. The structural rules of weakening remain visible in $\mathrm{LK}$ while they are absorbed in the block calculus, so that several derivations differing only by structural steps collapse onto one tableau. The order of application of the rules is free on both sides, and generates several derivations for one refutation. And in Gentzen's original formulation antecedents are lists, with contraction and exchange explicit, whereas blocks are sets that forget multiplicities and order -- a collapse that our set-based presentation of $\mathrm{LK}$ has already partly performed. A bijection would require normal forms on both sides: cut-free derivations free of explicit weakenings and with a fixed discipline on the order of principal formulae, and tableaux built under the same discipline. Whether that discipline can be given so as to make the bijection canonical we leave open. That the goal is attainable in a neighbouring case is known: von Plato \cite{vonPlato2001} exhibits an isomorphism between cut-free sequent derivations and normal natural deductions with general elimination rules, and it is that result which fixes the standard the present correspondence does not yet meet.
\end{rem}

\subsection{The block calculus and the $\mathrm{G3}$ family}\label{sec:G3}

The comparison with $\mathrm{LK}$ has a cost that is worth naming. In $\mathrm{LK}$ weakening is a primitive rule and the axioms are of the form $\alpha \segue \alpha$ without context; it is for this reason that the passage from $\mathrm{LK}$ to tableaux needs Lemmas~\ref{lem:rinomina}--\ref{lem:inversione}, which are pure bookkeeping about parameters and about the inversion of $\lnot$s. In a $\mathrm{G3}$-style calculus, with initial sequents $p, \Gamma \segue p, \Delta$ carrying an arbitrary context and with weakening and contraction height-preserving admissible \cite{NegriVonPlato2001}, those lemmas become citations rather than proofs. The invertible, context-sharing form of the operational rules on which such calculi rest goes back to Ketonen \cite{Ketonen1944}, and the simplification it brings to the elimination of cut is by now standard \cite{Buss1998, Franks2026}.

The affinity is not superficial. Section~\ref{sec:strutturale} establishes for blocks exactly the list of properties that characterizes the $\mathrm{G3}$ calculi, and in the same order of dependence: closure on literals with generalized closure derivable, height-preserving admissibility of weakening and of substitution, height-preserving invertibility of the rules, and admissibility of cut by double induction on the degree of the cut formula and on the sum of the heights. The block calculus is, from this angle, the empty-succedent fragment of a $\mathrm{G3}$-style calculus for classical logic read as a refutation procedure, and the persistence of the principal formula in the $\gamma$ rules is what stands in that fragment for height-preserving admissible contraction on the left. The propositional case of this identification, together with the corresponding natural deduction system and its normalization theorem, is developed in \cite{cuconato-proofteoria}.

\subsection{The translation at work}

We illustrate the translation on a propositional case, which carries over unchanged to first-order logic; the common tree format of the two displays is what makes the inversion of orientation visible. Consider $\varphi \lor \psi \segue \lnot\varphi \rightarrow \psi$, whose block tableau is

\begin{prooftree}
\AxiomC{$\varphi \lor \psi, \; \lnot(\lnot\varphi \rightarrow \psi)$ \quad $\Pi$}
\RightLabel{$\lnot\rightarrow$}
\UnaryInfC{$\varphi \lor \psi, \; \lnot\varphi, \; \lnot\psi$}
\RightLabel{$\lor$}
\UnaryInfC{$\{\varphi, \lnot\varphi, \lnot\psi\} \times \quad \mid \quad \{\psi, \lnot\varphi, \lnot\psi\} \times$}
\end{prooftree}

\noindent Replacing each block by the corresponding sequent and inverting the tree, one obtains the derivation

\begin{prooftree}
\AxiomC{$\varphi \segue \varphi, \psi$ \; $(\mathrm{Ax})^{\text{Wd}}$}
\RightLabel{$\lnot$s}
\UnaryInfC{$\lnot\varphi, \varphi \segue \psi$}
\AxiomC{$\psi \segue \psi, \varphi$ \; $(\mathrm{Ax})^{\text{Wd}}$}
\RightLabel{$\lnot$s}
\UnaryInfC{$\lnot\varphi, \psi \segue \psi$}
\RightLabel{$\lor$s}
\BinaryInfC{$\lnot\varphi, \varphi \lor \psi \segue \psi$}
\RightLabel{$\rightarrow$d}
\UnaryInfC{$\varphi \lor \psi \segue \lnot\varphi \rightarrow \psi$}
\end{prooftree}

\noindent The two trees have the same skeleton: a closed block becomes a basic sequent with possible weakening, a $\beta$ bifurcation becomes a two-premiss rule, an $\alpha$ or $\gamma$ step becomes a one-premiss rule. No step employs Cut. Table~\ref{tab:corrispondenza} records the dictionary item by item.

\begin{table}[htb]
{\small
\renewcommand{\arraystretch}{1.25}
\begin{tabular}{@{}p{48mm}@{\quad}c@{\quad}p{56mm}@{}}
\hline
\textbf{Block calculus} & & \textbf{Cut-free $\mathrm{LK}$} \\
\hline
block $\Pi=\Gamma\cup\lnot[\Delta]$ & $\longleftrightarrow$ & sequent $\Gamma\segue\Delta$, read one-sidedly as $\Pi\segue$ \\
closed block $\{p,\lnot p\}$ & $\longleftrightarrow$ & basic sequent $p\segue p$ (with $\lnot$s and weakening) \\
$\alpha$ step (one child) & $\longleftrightarrow$ & one-premiss rule (via $\lnot$s and its inversion) \\
$\beta$ step (bifurcation) & $\longleftrightarrow$ & two-premiss operational rule \\
$\gamma$ step, principal formula retained & $\longleftrightarrow$ & $\forall$s\,/\,$\exists$d with left contraction \\
$\delta$ step, fresh constant $\aup$ & $\longleftrightarrow$ & critical $\exists$s\,/\,$\forall$d, proper parameter \\
set structure of blocks, closure criterion & $\longleftrightarrow$ & weakening and contraction, absorbed \\
admissibility of cut (Thm~\ref{thm:taglio}) & $\longleftrightarrow$ & \emph{Hauptsatz} (Cor.~\ref{cor:hauptsatz}) \\
\hline
\end{tabular}}
\caption{The dictionary underlying Theorem~\ref{thm:corrispondenza}: a closed
tableau for $\Pi$ and a cut-free derivation of $\Gamma\segue\Delta$ are the same
tree read in opposite orientations. The last line records the transfer of
Corollary~\ref{cor:hauptsatz}.}
\label{tab:corrispondenza}
\end{table}

\section{Examples of derivation}\label{sec:esempi}

Four derivations follow. They are chosen for what they exhibit about the interaction of the quantifier rules, and not for the difficulty of the sequents they establish: an open branch and the model it describes, the reiteration of a $\gamma$ rule on a constant introduced later by a $\delta$ rule, the discipline of order among witnesses under nested quantifiers, and a construction that does not terminate, whose root fails the hypothesis of Theorem~\ref{thm:terminazione}.

\begin{con}[Derivation displays]\label{conv:display}
A derivation is displayed as a tree with the root block at the top; each rule application is drawn as an inference step whose label, to the right of the line, records the rule and, for the quantifier rules, the parameter chosen. Every node carries the entire current block, written out in full: no formula is elided and no ellipsis is employed. A terminal block is enclosed in braces and marked $\times$ if closed, $\odot$ if completed open, and a $\beta$ rule displays its two children side by side, separated by $\mid$. The convention is heavier than the usual one, in which contexts are abbreviated, and it is adopted deliberately: sequent derivations grow broad, and it is precisely the elided context that hides the side conditions on parameters, as von Plato observes in a related connection \cite{vonPlato2017}.
\end{con}

Throughout, $a_0$ is the reserved parameter of Definition~\ref{def:parametri} and the constants introduced by $\delta$ rules are named $b, c, d, \ldots$

\begin{exm}[Invalidity and openness of the block]\label{es:aperto}
The sequent $\exists x\, P(x) \segue \forall x\, P(x)$ is not valid.
\end{exm}
\begin{proof}
The associated block is $\Pi = \{\exists x\, P(x),\ \lnot\forall x\, P(x)\}$.
\begin{prooftree}
\AxiomC{$\exists x\, P(x), \; \lnot \forall x\, P(x)$ \quad $\Pi$}
\RightLabel{$\exists$, \; $\aup=b$}
\UnaryInfC{$\lnot \forall x\, P(x), \; P(b)$}
\RightLabel{$\lnot\forall$, \; $\aup=c$, \; $c \neq b$}
\UnaryInfC{$\{P(b), \; \lnot P(c)\} \; \odot$}
\end{prooftree}
Both formulae of $\Pi$ are of type $\delta$ and require, by Definition~\ref{def:nuova}, two distinct constants: reusing $b$ is not permitted (Remark~\ref{oss:novita}). The terminal block contains no complementary pair, $P(b)$ and $\lnot P(c)$ being distinct atoms, and it contains only literals, so no rule applies to it: the block is completed open, and the tableau displayed is fair. By Theorem~\ref{thm:completezza}, read in contraposition, $\Pi$ is satisfiable. The Hintikka model of the open branch confirms it: $D = \mathrm{Par}(\mathcal R) = \{a_0, b, c\}$ with each constant interpreted as itself and $i(P) = \{b\}$ satisfies $\exists x\,P(x)$ and falsifies $\forall x\,P(x)$.
\end{proof}

\begin{exm}[The drinker paradox]\label{es:bevitore}
The sequent $\segue \exists x\,(P(x) \rightarrow \forall y\, P(y))$ is valid.
\end{exm}
\begin{proof}
The associated block is $\Pi = \{\lnot\exists x\,(P(x)\rightarrow\forall y\,P(y))\}$. Since $\Pi$ contains no individual constant, the first $\gamma$ instantiation takes place on the reserved parameter $a_0$.
\begin{prooftree}
\AxiomC{$\lnot\exists x\,(P(x) \rightarrow \forall y\, P(y))$ \quad $\Pi$}
\RightLabel{$\lnot\exists$, \; $\adown=a_0$}
\UnaryInfC{$\lnot\exists x\,(P(x) \rightarrow \forall y\, P(y)), \; \lnot(P(a_0) \rightarrow \forall y\, P(y))$}
\RightLabel{$\lnot\rightarrow$}
\UnaryInfC{$\lnot\exists x\,(P(x) \rightarrow \forall y\, P(y)), \; P(a_0), \; \lnot\forall y\, P(y)$}
\RightLabel{$\lnot\forall$, \; $\aup=b$}
\UnaryInfC{$\lnot\exists x\,(P(x) \rightarrow \forall y\, P(y)), \; P(a_0), \; \lnot P(b)$}
\RightLabel{$\lnot\exists$, \; $\adown=b$}
\UnaryInfC{$\lnot\exists x\,(P(x) \rightarrow \forall y\, P(y)), \; P(a_0), \; \lnot P(b), \; \lnot(P(b) \rightarrow \forall y\, P(y))$}
\RightLabel{$\lnot\rightarrow$}
\UnaryInfC{$\{\lnot\exists x\,(P(x) \rightarrow \forall y\, P(y)), \; P(a_0), \; \lnot P(b), \; P(b), \; \lnot\forall y\, P(y)\} \times$}
\end{prooftree}
The $\gamma$ formula is instantiated twice: first on $a_0$, then on the constant $b$ that the $\delta$ rule has introduced meanwhile. The second instantiation is the one that closes the branch. It places the negated implication on $b$ into the block, and the rule $(\lnot\rightarrow)$ extracts from it $P(b)$, the complement of the $\lnot P(b)$ already present. That the same application reintroduces $\lnot\forall y\, P(y)$, which the $\delta$ rule had consumed, is harmless, blocks being sets. Two features of the calculus are indispensable here: the persistence of the $\gamma$ formula, and the obligation (F2) to instantiate it on every parameter of the branch, those generated later included.
\end{proof}

\begin{exm}[Nested quantifiers: the valid direction]\label{es:annidati-si}
The sequent
\[
\exists x\,\forall y\, R(x,y) \segue \forall y\,\exists x\, R(x,y)
\]
is valid.
\end{exm}
\begin{proof}
The associated block is $\Pi = \{\exists x\,\forall y\, R(x,y),\ \lnot\forall y\,\exists x\, R(x,y)\}$.
\begin{prooftree}
\AxiomC{$\exists x\,\forall y\, R(x,y), \; \lnot\forall y\,\exists x\, R(x,y)$ \quad $\Pi$}
\RightLabel{$\exists$, \; $\aup=b$}
\UnaryInfC{$\forall y\, R(b,y), \; \lnot\forall y\,\exists x\, R(x,y)$}
\RightLabel{$\lnot\forall$, \; $\aup=c$}
\UnaryInfC{$\forall y\, R(b,y), \; \lnot\exists x\, R(x,c)$}
\RightLabel{$\forall$, \; $\adown=c$}
\UnaryInfC{$\forall y\, R(b,y), \; \lnot\exists x\, R(x,c), \; R(b,c)$}
\RightLabel{$\lnot\exists$, \; $\adown=b$}
\UnaryInfC{$\{\forall y\, R(b,y), \; \lnot\exists x\, R(x,c), \; R(b,c), \; \lnot R(b,c)\} \times$}
\end{prooftree}
The order is essential here, and it obeys the discipline by which a $\delta$ rule must precede the $\gamma$ rules that will draw on its constant. The two $\delta$ rules consume their principal formulae and introduce the distinct witnesses $b$ and $c$. Only then can the two $\gamma$ rules instantiate, on $c$ and on $b$ respectively, and produce the complementary pair.
\end{proof}

\begin{exm}[Nested quantifiers: the invalid direction]\label{es:annidati-no}
The sequent
\[
\forall y\,\exists x\, R(x,y) \segue \exists x\,\forall y\, R(x,y)
\]
is not valid.
\end{exm}
\begin{proof}
The associated block is $\Pi = \{\forall y\,\exists x\, R(x,y),\ \lnot\exists x\,\forall y\, R(x,y)\}$, both of whose formulae are of type $\gamma$; since $\Pi$ contains no constant, the construction starts from $a_0$. One fair schedule begins as follows.
\begin{prooftree}
\AxiomC{$\forall y\,\exists x\, R(x,y), \; \lnot\exists x\,\forall y\, R(x,y)$ \quad $\Pi$}
\RightLabel{$\forall$, \; $\adown=a_0$}
\UnaryInfC{$\forall y\,\exists x\, R(x,y), \; \lnot\exists x\,\forall y\, R(x,y), \; \exists x\, R(x,a_0)$}
\RightLabel{$\exists$, \; $\aup=b$}
\UnaryInfC{$\forall y\,\exists x\, R(x,y), \; \lnot\exists x\,\forall y\, R(x,y), \; R(b,a_0)$}
\RightLabel{$\lnot\exists$, \; $\adown=b$}
\UnaryInfC{$\forall y\,\exists x\, R(x,y), \; \lnot\exists x\,\forall y\, R(x,y), \; R(b,a_0), \; \lnot\forall y\, R(b,y)$}
\RightLabel{$\lnot\forall$, \; $\aup=c$}
\UnaryInfC{$\forall y\,\exists x\, R(x,y), \; \lnot\exists x\,\forall y\, R(x,y), \; R(b,a_0), \; \lnot R(b,c)$}
\RightLabel{$\forall$, \; $\adown=c$}
\UnaryInfC{$\forall y\,\exists x\, R(x,y), \; \lnot\exists x\,\forall y\, R(x,y), \; R(b,a_0), \; \lnot R(b,c), \; \exists x\, R(x,c)$}
\RightLabel{$\exists$, \; $\aup=d$}
\UnaryInfC{$\forall y\,\exists x\, R(x,y), \; \lnot\exists x\,\forall y\, R(x,y), \; R(b,a_0), \; \lnot R(b,c), \; R(d,c)$}
\noLine
\UnaryInfC{$\vdots$}
\end{prooftree}
The last block displayed contains the atoms $R(b,a_0)$, $\lnot R(b,c)$, $R(d,c)$, no two of which are complementary, since $b \neq d$ and $a_0 \neq c$. Nor can the construction halt: each $\delta$ application introduces a new parameter on which (F2) requires the two $\gamma$ formulae to be instantiated anew, and each instantiation of $\forall y\,\exists x\,R(x,y)$ generates a further existential, hence a further witness, while each instantiation of $\lnot\exists x\,\forall y\,R(x,y)$ generates a further negated universal. The branch grows indefinitely. The root, incidentally, is not $\delta$-stratified: the $\gamma$ formula $\forall y\,\exists x\,R(x,y)$ has the $\delta$ formula $\exists x\,R(x,y)$ among its subformulae, so Theorem~\ref{thm:terminazione} does not apply, and the example shows that its hypothesis cannot simply be dropped.

The display of an initial segment does not prove that no tableau for $\Pi$ closes; that is established semantically. Let $M = (D,i)$ with $D = \mathbb{N}$ and $i(R) = \{(m,n) \mid m > n\}$. Then $M \vDash \forall y\,\exists x\, R(x,y)$, the witness for $n$ being $n+1$, and $M \nvDash \exists x\,\forall y\, R(x,y)$, no natural number exceeding all the others, itself included. Hence $\Pi$ is satisfiable, and by Theorem~\ref{thm:correttezza} read in contraposition no tableau for $\Pi$ can close.
\end{proof}

\section{Concluding remarks}\label{sec:conclusioni}

The calculus of sequent-style tableaux, born in the propositional setting from the need to combine the readability of the sequent calculus with the refutational effectiveness of tableaux \cite{cuconato2025metodi}, extends to first-order logic without strain. Each node carries a block $\Pi$; each rule decomposes the principal formula and preserves the rest, save for the $\delta$ rules, which replace it by an instance on a fresh constant; a branch closes as soon as a complementary pair of literals appears. The addition of the $\gamma$ and $\delta$ rules requires that one account for the freshness condition, in its twofold syntactic and semantic reading, and for the non-reductive character of universal instantiation. Against this cost, the calculus gains the expressiveness of the language of predicates while remaining sound and complete for the semantics of Section~\ref{sec:preliminari}, analytic in the sense of Corollary~\ref{cor:analiticita}, and closed under cut.

That last point deserves to be separated from the others. A calculus without a cut rule is not thereby a calculus in which cut is superfluous, and the difference is the whole content of Gentzen's theorem. Theorem~\ref{thm:taglio}, which is Smullyan's, settles the matter on the side of blocks by a finitary argument, and Corollary~\ref{cor:hauptsatz} then returns the Hauptsatz for $\mathrm{LK}$ by transfer: a derivation with cut is translated into a tableau, where cut is absorbed, and the tableau is translated back into a derivation, which is cut-free because tableaux have nothing else to offer. It is here that the correspondence of Section~\ref{sec:LK} earns its place; without it, the two calculi would merely prove the same theorems.

On the plane of pure derivability the two systems are, as Theorem~\ref{thm:corrispondenza} shows, one combinatorial structure. Differences appear as soon as one moves beyond what is derivable. The sequent calculus is a calculus of synthesis: it builds the derivation upwards, and its natural reading is that of a proof being erected. The block calculus is a calculus of analysis and of search: it starts from the negation of what one wishes to establish and decomposes, in the attempt to exhibit a model, and when the attempt fails on every branch that failure \emph{is} the proof. The two directions describe distinct cognitive acts, even though the underlying structure is the same. A second difference concerns what each system makes explicit: in $\mathrm{LK}$ weakening is a visible step, in the block calculus it is absorbed into the data structure and into the closure criterion, and the choice of what to display and what to hide directs both the reader's attention and the measure of the complexity of derivations. What Hazen and Pelletier observed of Gentzen's and Ja\'skowski's natural deduction \cite{HazenPelletier2014} holds here too: two formulations that prove the same theorems may be importantly different once they are tested on ground for which neither was designed.

Analyticity, finally, binds the syntax to the semantics. Since no formula extraneous to the root can appear in a refutation, a branch that fails to close gathers exactly the formulae needed to construct a model; this is what Lemma~\ref{lem:hintikka} and Theorem~\ref{thm:completezza} establish and what Example~\ref{es:annidati-no} shows at work. To close every branch is to prove, to leave one open is to construct a model; it is this bridge between the proof and the model that it seeks or excludes that structural proof theory takes as its core \cite{NegriDirezioni, Poggiolesi2012}. It should not be concealed that analyticity has a price -- cut-free proofs represent mathematical reasoning, which composes lemmas rather than decomposing theses, imperfectly, and they may be very much longer \cite{Boolos1984, Statman1978} -- and it is in response to this that the line of analytic cut was developed \cite{DAgostinoMondadori1994}. Corollary~\ref{cor:taglio-analitico} shows where the block calculus stands on that question: analytic cut is available, and it is available as a theorem rather than as an addition to the rules.

It remains to say what is not attempted. Interpolation is not: Craig's lemma and Beth's theorem are obtained from block tableaux, through a Gentzen system with a strengthened subformula property, in \cite[Ch.~XIV--XV]{Smullyan1968}, and we have nothing to add to that treatment. What is left open is narrower, and it is the question of Remark~\ref{oss:non-iso}: which discipline on the order of principal formulae turns the simulation of Theorem~\ref{thm:corrispondenza} into a bijection, in the sense in which von Plato \cite{vonPlato2001} obtains one between cut-free sequent derivations and normal natural deductions with general elimination rules. The other direction we regard as more promising, and it lies outside this paper: the network of the companion work \cite{cuconato-proofteoria}, which relates the propositional block calculus to the sequent calculi $\mathrm{G0T}$ and $\mathrm{G3T}$ and to a natural deduction system with general elimination rules, has no first-order counterpart, and the normalization theorem on which it ends has not been extended to quantification theory.

\bibliographystyle{amsplain}

\begin{thebibliography}{30}

\bibitem{Boolos1984}
G. Boolos,
\emph{Don't eliminate cut},
J. Philos. Logic \textbf{13} (1984), 373--378.

\bibitem{Buss1998}
S.\,R. Buss,
\emph{An introduction to proof theory},
in: S.\,R. Buss (ed.), \emph{Handbook of Proof Theory}, North-Holland, Amsterdam, 1998, 1--78.

\bibitem{Church1936}
A. Church,
\emph{An unsolvable problem of elementary number theory},
Amer. J. Math. \textbf{58} (1936), 345--363.

\bibitem{cuconato2025metodi}
S. Cuconato,
\emph{Metodi Logici 1. Teoria degli Insiemi, Logica del Primo Ordine, Tableaux Analitici, Calcolo dei Sequenti},
Il Sileno Edizioni, Lago (CS), 2025.

\bibitem{cuconato-proofteoria}
S. Cuconato,
\emph{Proof theory for sequent-style tableaux: $\mathrm{G0}$- and $\mathrm{G3}$-style sequent calculi and full normalization},
manuscript.

\bibitem{DAgostinoMondadori1994}
M. D'Agostino and M. Mondadori,
\emph{The taming of the cut. Classical refutations with analytic cut},
J. Logic Comput. \textbf{4} (1994), 285--319.

\bibitem{Fitting1996}
M. Fitting,
\emph{First-Order Logic and Automated Theorem Proving},
2nd ed.,
Springer, New York, 1996.

\bibitem{Fitting1999}
M. Fitting,
\emph{Introduction},
in: M. D'Agostino, D.\,M. Gabbay, R. H\"ahnle, J. Posegga (eds.), \emph{Handbook of Tableau Methods}, Kluwer, Dordrecht, 1999, 1--43.

\bibitem{Franks2026}
C. Franks,
\emph{Gentzen's Logical Calculi (the Theory Pamphlet)},
SpringerBriefs Philos., Springer, Cham, 2026.

\bibitem{Galvan2015}
S. Galvan,
\emph{L'Hauptsatz di Gentzen},
EDUCatt, Milano, 2015.

\bibitem{Gentzen1935}
G. Gentzen,
\emph{Untersuchungen \"uber das logische Schlie{\ss}en I--II},
Math. Z. \textbf{39} (1935), 176--210, 405--431.

\bibitem{HazenPelletier2014}
A.~P. Hazen and F.~J. Pelletier,
\emph{Gentzen and Ja\'skowski natural deduction: Fundamentally similar but importantly different},
Stud. Log. \textbf{102} (2014), 1103--1142.

\bibitem{Ketonen1944}
O. Ketonen,
\emph{Untersuchungen zum Pr\"adikatenkalk\"ul},
Ann. Acad. Sci. Fenn. Ser. A I Math.-Phys. \textbf{23} (1944).

\bibitem{Konig1927}
D. K\"onig,
\emph{\"Uber eine Schlussweise aus dem Endlichen ins Unendliche},
Acta Sci. Math. (Szeged) \textbf{3} (1927), 121--130.

\bibitem{mancosu2021}
P. Mancosu, S. Galvan, and R. Zach,
\emph{An Introduction to Proof Theory: Normalization, Cut-Elimination, and Consistency Proofs},
Oxford University Press, Oxford, 2021.

\bibitem{NegriDirezioni}
S. Negri,
\emph{Teoria strutturale della dimostrazione},
in \emph{Le direzioni della ricerca logica in Italia},
H. Hosni, G. Lolli, and C. Toffalori (eds.),
Edizioni della Normale, Pisa, 2015,
221--253.

\bibitem{NegriVonPlato2001}
S. Negri and J. von Plato,
\emph{Structural Proof Theory},
Cambridge University Press, Cambridge, 2001.

\bibitem{Palladino2002}
D. Palladino,
\emph{Corso di logica. Introduzione elementare al calcolo dei predicati},
Carocci, Roma, 2002.

\bibitem{vonPlato2001}
J. von Plato,
\emph{Natural deduction with general elimination rules},
Arch. Math. Logic \textbf{40} (2001), 541--567.

\bibitem{vonPlato2017}
J. von Plato,
\emph{From Gentzen to Jaskowski and back: algorithmic translation of derivations between the two main systems of natural deduction},
Bull. Sect. Log. \textbf{46} (2017), 65--73.

\bibitem{Poggiolesi2012}
F. Poggiolesi,
\emph{On the importance of being analytic. The paradigmatic case of the logic of proofs},
Log. Anal. (N.S.) \textbf{55} (2012), 443--461.

\bibitem{Schuette1977}
K. Sch\"utte,
\emph{Proof Theory},
Springer, Berlin, 1977.

\bibitem{Simpson2009}
S.~G. Simpson,
\emph{Subsystems of Second Order Arithmetic},
2nd ed.,
Cambridge University Press, Cambridge, 2009.

\bibitem{Smullyan1968}
R.~M. Smullyan,
\emph{First-Order Logic},
Springer, Berlin, 1968.

\bibitem{Statman1978}
R. Statman,
\emph{Bounds for proof-search and speed-up in the predicate calculus},
Ann. Math. Logic \textbf{15} (1978), 225--287.

\bibitem{Tait1968}
W.~W. Tait,
\emph{Normal derivability in classical logic},
in \emph{The Syntax and Semantics of Infinitary Languages},
J. Barwise (ed.),
Lecture Notes in Math. 72,
Springer, Berlin, 1968,
204--236.

\bibitem{troelstra2000}
A.~S. Troelstra and H. Schwichtenberg,
\emph{Basic Proof Theory},
2nd ed.,
Cambridge University Press, Cambridge, 2000.

\bibitem{Turing1936}
A.~M. Turing,
\emph{On computable numbers, with an application to the Entscheidungsproblem},
Proc. London Math. Soc. (2) \textbf{42} (1936), 230--265;
correction, ibid. (2) \textbf{43} (1937), 544--546.

\bibitem{VanDalen2013}
D. van Dalen,
\emph{Logic and Structure},
5th ed.,
Springer, London, 2013.

\end{thebibliography}

\end{document}